\documentclass[11pt]{article} 
\usepackage{apacite}
\usepackage[affil-it]{authblk}
\usepackage{setspace}
\usepackage[utf8]{inputenc} 

\usepackage{subfigure}
\usepackage[margin=1in]{geometry}
\usepackage{graphicx} 
\usepackage{epstopdf}

\usepackage{booktabs} 
\usepackage{array} 
\usepackage{paralist} 
\usepackage{verbatim} 
\usepackage{subfig} 

\usepackage{fancyhdr} 
\usepackage{sectsty}
\allsectionsfont{\sffamily\mdseries\upshape} 

\usepackage[nottoc,notlof,notlot]{tocbibind} 
\usepackage[titles,subfigure]{tocloft} 

\usepackage{amsmath, amsthm, amssymb}

\usepackage{multirow}
\usepackage{algorithm}
\usepackage{algorithmic}

\usepackage{amsthm}

\newtheorem{prop}{Proposition}

\newtheorem{dfn}{Definition}

\allowdisplaybreaks

\usepackage{rotating}
\usepackage{pdflscape}

\usepackage{tikz,pgfplots}
\usetikzlibrary{matrix}
\usepgfplotslibrary{groupplots}
\pgfplotsset{compat=newest}
\usetikzlibrary{calc,positioning,shapes,fit,arrows,shadows}
\tikzstyle{line} = [ draw, -latex']
\usetikzlibrary{shapes,arrows}
\usepgfplotslibrary{units}
\usepackage{url}
\usepackage{color}

\graphicspath{{figures/}} 

\usepackage{color}

\usepackage{color}

\def\black{\color{black}}

\DeclareMathOperator*{\VaR}{VaR}
\DeclareMathOperator*{\CVaR}{CVaR}

\date{ }

\title{\vspace{-1cm}
Incorporating Black-Litterman Views \\ in Portfolio Construction \\ when Stock Returns are a Mixture of Normals
}

\author{Burak Kocuk\thanks{Corresponding author. burakkocuk@sabanciuniv.edu, Industrial Engineering Program,  Sabanc{\i} University, Istanbul, Turkey 34956. }  \ and 
G\'erard Cornu\'ejols\thanks{gc0v@andrew.cmu.edu, Tepper School of Business, Carnegie Mellon University, Pittsburgh, PA 15213 USA.
}
}

\begin{document}

\maketitle

\vspace{-1cm}
\begin{abstract}
In this paper, we consider the basic problem of portfolio construction in financial engineering, and analyze how market-based and analytical approaches can be combined to obtain efficient portfolios.
As a first step in our analysis, we model the asset returns as a random variable distributed according to a mixture of normal random variables.
We then discuss how to construct portfolios that minimize the Conditional Value-at-Risk (CVaR) under this probabilistic model via a convex program. We also construct a second-order cone representable approximation of the CVaR under the mixture model, and demonstrate its theoretical and empirical accuracy.
Furthermore, we incorporate the market equilibrium information into this procedure through the well-known Black-Litterman approach via an inverse optimization framework {by utilizing the proposed approximation}.
Our computational experiments on a real dataset show that this approach with an emphasis on the market equilibrium typically yields less risky portfolios than a purely market-based portfolio while producing similar returns on average.

\end{abstract}

\section{Introduction}

Portfolio construction is one of the most fundamental problems in financial engineering: Given $n$ risky assets, historical information about their returns and market capitalization of these assets, construct a  portfolio that will produce the maximum expected return at a minimum risk. Maximizing expected return and minimizing risk are often conflicting objectives, and hence a compromise should be made by investors based on their risk aversion.

We  consider two paradigms for solving this problem: an ``analytical" approach and one that is ``market-based". In the analytical approach, key parameters such as the vector of mean asset returns, denoted by $\mu$, and covariance between the asset returns, denoted by $\Sigma$, are estimated from historical data. Then, a combination of the expected portfolio return and a risk measure is optimized by solving a problem of the form:
\begin{equation}\label{eq:portfolio generic}
\max_{x \in \mathcal{X}} \  \{ \mu^T x - \delta  {R}(x) \}.
\end{equation}
Here,  ${R}(x)$ is the risk of portfolio $x$ under a chosen risk measure, $\delta$ is a positive, predetermined risk aversion factor and $\mathcal{X}$ is the set of feasible portfolios. The earliest example of this approach is 
\citeA{markowitz1952}, in which a mean-variance (MV) optimization problem{\black\footnote{\black Although  variance and standard deviation are technically  deviation measures, by adapting a slight abuse of terminology, we will be referring to them as risk measures as well.}} is proposed with $n=3$ assets, ${R}(x) = x^T \Sigma x$ and $\mathcal{X} = \Delta_n :=\{x\in\mathbb{R}^n: \ \sum_{j=1}^n x_j = 1, \ x\ge 0\}$.

In the market-based approach, one merely invests in a ``market portfolio" proportional to the current market capitalization of the assets. The logic behind this market-based approach is the Efficient-Market Hypothesis \cite{fama1970}, which loosely states that the price of an asset captures all the information about that asset. 

There are advantages and disadvantages to each approach. The major advantage of the analytical approach is that if the parameter predictions are accurate, then it can yield provably optimal portfolios. Unfortunately, this almost never happens in practice.
 In particular, the estimation of the mean return vector $\mu$ is error-prone and
 even a small perturbation in the parameter estimation can yield completely different portfolios due to what is called the ``error-maximization property" in \citeA{Michaud1989}. Robust optimization techniques are proposed to circumvent the difficulties caused by parameter estimation in \citeA{goldfarb2003, ghaoui2003, tutuncu2004, ceria2006, demiguel2009, zhu2009worst}.
Another issue with the analytical approach is the determination of the risk aversion parameter~$\delta$. Choosing smaller values of $\delta$ puts more emphasis on the expected return and since the estimates of $\mu$ are generally inaccurate, it may lead to poor portfolios in practice.
One may alleviate this issue by simply choosing $\delta $ infinitely large, thus reducing the generic portfolio optimization \eqref{eq:portfolio generic} to a ``risk minimizing" portfolio problem (see, for instance,  \citeA{demiguel2009}). 

Another issue regarding the analytical approach is to specify an adequate risk measure. Although variance (or standard deviation) is a typical risk measure, others are also used, for instance, Value-at-Risk (VaR) and Conditional Value-at-Risk (CVaR), which can better quantify the downside risk. One difficulty with these measures is that {\black either} a distributional assumption should be made, which requires modeling stock returns as random variables, {\black or sampling-based simulation methods should be utilized based on the historical observations}. A first choice is a multivariate normal distribution, which allows easy-to-solve risk minimization problems for both VaR and CVaR measures. However,  the normal distribution typically does not provide a good fit to the stock return data due to heavy tails. Other probabilistic models have been proposed including stable distribution \cite{mandelbrot1963}, $t$-distribution \cite{blattberg1974},  and mixture of multivariate normals \cite{chin1999computing, buckley2008, chen2013novel, wang2015} among others.
As the probabilistic model becomes more complex, it might be difficult to solve optimization problems involving VaR or CVaR terms. {\black In such cases, return scenarios can be generated through Monte Carlo simulation to estimate these measures without any distributional assumption \cite{krokhmal2002}}.


The main advantage of the market-based approach is   its simplicity since it does not require any parameter estimation, optimization or distribution fitting procedure. One can, for instance, simply track Standard \& Poor's 500 index, which is arguably a representative proxy for the United States stock market. The disadvantage of  the market-based approach is its inflexibility. For instance, if an investor believes that a certain stock will outperform another, this approach does not allow to incorporate this view.


Several studies try to combine the two approaches. For instance, in the early 1990s, \citeA{black1991, black1992} proposed a way of combining the market portfolio and investor views into the classical MV approach. In practice,  portfolios obtained using the Black-Litterman (BL)   methodology tend to be more robust to data perturbations. But there are also some issues with the BL methodology: For instance, the derivation of the estimates is not very intuitive and a number of papers including \citeA{he1999, satchell2000, drobetz2001avoid, meucci2010black} attempt to clarify it from different perspectives.  Moreover, the BL derivation is based on strong assumptions, one of them being the normality of the random return vector.  Furthermore, parameters have to be determined exogenously to incorporate the confidence in the  investor views.

Recently, the BL model has been generalized using an  interpretation as an inverse optimization problem in \citeA{bertsimas2012}. Different from the classical derivation of the BL model, mean and covariance of the returns are determined  as the solution of a conic program. This approach is flexible enough to eliminate some of the shortcomings of the BL methodology such as its inability to include views on variance, necessity to exogenously choose several parameters etc.
Other recent extensions of the BL methodology include \citeA{jia2016} in which inverse optimization incorporates views on the variance, 
{\black \citeA{pang2018} in which a closed-form solution is derived when the risk measure is chosen as CVaR and the stock returns are assumed to follow an elliptical distribution,}  
and \citeA{silva2017} in which views are created using Verbal Decision Analysis. 

In this paper, we address several issues raised above and extend recent work. We start our analysis by focusing on the vector of stock returns modeled as a random variable. 
Using Standard \& Poor's (S\&P) 500 dataset, we show that the returns are not normally distributed via statistical tests, and we propose an enhanced probabilistic model, namely a mixture of normal random variables. 
Then, we discuss how to construct risk minimizing portfolios under different probabilistic models (normal and mixture of normals) and risk measures (standard deviation and CVaR). We also propose a BL-type approach, which  incorporates market-based information into CVaR minimizing portfolios.

\citeA{buckley2008} model the stock returns as a mixture of two normal random variables, considering several objective functions for the mean and variance of each ``regime". The main difference in our work is that we optimize portfolios  directly with respect to the CVaR of a mixture distribution {\black via a convex program.}


Our key contributions in this paper are summarized below:
\begin{enumerate}[1.]

\item
We analyze how to construct portfolios that minimize CVaR under the mixture model. Although CVaR minimization under the normal distribution is straightforward, resulting in a second-order cone program, the case with the mixture is less  obvious since the CVaR of a mixture distribution does not have a closed form expression. We analyze how it can be numerically computed {\black and optimized via a convex program}. We also propose a closed form, second-order cone representable approximation of CVaR in this~case.

\item
We extend the work on the BL approach via inverse optimization to CVaR minimization  under both normal and mixture distribution models. In the latter case, we propose a sophisticated approach, which combines our previous two contributions. In particular, our model is governed by the market equilibrium equation,
   and  the parameters of the mixture distribution are treated as  investor views.

\item
We present computational results applied to the S\&P 500 dataset. Empirically, we observe that,
as expected, market-based portfolios typically have higher reward and higher risk than risk minimizing portfolios.
However, we show on the same dataset that a certain combination of the market-based and risk minimizing portfolios obtained through a BL-type approach may yield portfolios with similar rewards and smaller risk.
\end{enumerate}

The rest of the paper is organized as follows:
In Section~\ref{sec:stat analy}, we provide a statistical analysis of the S\&P 500 dataset and model the stock returns as a mixture of normal random variables.
In Section~\ref{sec:risk-min},  we present our portfolio optimization problem with different risk measures.
In Section~\ref{sec:combine}, we propose a new approach to combine CVaR minimizing portfolios with market information to obtain BL-type portfolios.
In Section~\ref{sec:comp}, we present our computational experiments on the S\&P 500 dataset, and compare market-based, risk minimizing and combination portfolios.
Finally,  Section~\ref{sec:conclusion} concludes our paper with further remarks. 

\section{Statistical Analysis of Stock Returns}
\label{sec:stat analy}


As opposed to the standard Markowitz approach, which does not require any distribution information on the asset returns to construct the portfolios, the VaR and CVaR measures need  {\black either} explicit forms of the distributions, {\black or the use of sampling based methods}. 
It is not uncommon in the finance literature to model the stock returns as normally distributed random variables. However, stock returns are rarely normally distributed, and typically have heavier tails. Therefore, it is crucial to use a different probabilistic model to capture the
heavy tail effect, especially the left tail, which is closely related to the risk of a portfolio when the VaR and CVaR measures are used.

{\black In this section, we use a real dataset, specifically the stocks in Standard \& Poor's (S\&P) 500 index over a 30-year time span. Since the statistical evidence suggest that the stock returns do not come from a normal distribution, we propose an alternative probabilistic model, namely, we model the stock returns as a mixture of normal random variables  and explain how the parameters of the mixture distribution can be estimated.}


\subsection{Data Collection and Normality Tests}
\label{sec:data coll}
We first collected historical stock returns and market capitalization information from the Wharton Research Database Services (WRDS).  Since working with tens of thousands of different stocks is not
appropriate for this study, we focused on stocks in the S\&P 500 index. Following~\citeA{bertsimas2012}, we further simplified our analysis by focusing on 11 sectors according to the Global Industrial Classification Standard (GICS). 
{ As a consequence of this simplification, we do not need to keep track of the assets that enter or leave the S\&P 500, rather we concentrate on the overall performance of each sector as the average performance of its constituents.}
Using WRDS, we collected the return and market capitalization information for all the stocks that have been in the S\&P 500 between January 1987-December 2016, spanning a 30-year period. We then computed the return of a sector for each month as a weighted average of the returns of the companies in S\&P 500  in that particular time period, where the weights are taken as the market capitalization of each stock in that sector. This procedure gave us 360 sector return vectors of size 11, denoted by $R^t$, $t=1,\dots, 360$. We also recorded the percentage market capitalization of each sector $j$ in month $t$, denoted by $M_j^t$, $j=1,\dots,11$, $t=1,\dots,360$.

{
We use an \texttt{R} package called \texttt{MVN} \cite{korkmaz} to formally test the multivariate normality of the sector return vectors. 
We also test whether the returns of individual sectors are normally distributed.
According to our extensive tests, 
we conclude that there is significant evidence that the neither the sector return vector nor the returns of individual sectors are normally distributed as expected.
}

\subsection{Modeling Returns as Mixture of Normal Random Variables}
\label{sec:two models}
The fact that the vector of sector returns is unlikely to come from a multivariate normal distribution motivates us to search for an alternate probabilistic model that better explains the randomness of the stock returns. We will try to construct such a model using mixtures of (multivariate) normal random variables.

This choice for our model can be explained from two perspectives.
First, we note that the stock returns have typically heavier left tails, which can be considered as the most critical part since it directly relates to the investment risk. This was previously observed by many researchers, including  the J.P. Morgan Asset Management group \cite{sheikh2010non}. In this paper, we try to capture this effect by introducing a mixture of random variables.
An intuitive explanation for this phenomenon is offered as follows: Under ``regular"  conditions, the market, in fact, behaves following an approximate normal distribution. However, every once in awhile, a ``shock" happens and shifts the mean of the return distribution to the left with possibly higher variance. This can explain the relatively heavier left tails of the empirical distribution.
 Second, as we will demonstrate below, introducing this more sophisticated probabilistic model greatly improves the fit to the data.  However, this better fit comes with a cost of more complicated procedures for data fitting and for portfolio optimization. The data fitting issue will be addressed in the remainder of this section. As for the portfolio optimization procedures under these more complicated distributions, they will be discussed  in Section \ref{sec:mixture cvar opt}.

Formally, let us assume that the random return is distributed as a mixture of two multivariate normal random variables, that is, with some probability $\rho_i$, returns are normally distributed with mean $\mu^i$ and covariance matrix $\Sigma^i$, for $i=1,2$.
In other words, we have
\begin{equation*}
r_M = \begin{cases}  r_{M,1} & \text{w.p. } \rho_1 \\r_{M,2} & \text{w.p. } \rho_2 \end{cases}
\text{ where } r_{M,i} \sim \text{N}(\mu^i, \Sigma^i), \ i=1,2.
\end{equation*}
Note that if $\rho_i$, $\mu^i$, $\Sigma^i$, $i=1,2$,  are given, we can compute the expectation and covariance of $r_M$ as follows:
\begin{equation}\label{eq: r mean var}
\text{E}[r_M] = \rho_1 \mu^1 + \rho_2 \mu^2 \text{ and } \text{Cov}(r_M) = \rho_1 \Sigma^1+ \rho_2 \Sigma^2 + \rho_1 \rho_2 (\mu^1-\mu^2)(\mu^1-\mu^2)^T.
\end{equation}
A direct approach to estimate the parameters of the mixture distribution is to use maximum-likelihood estimation, which can be achieved by the Expectation-Maximization (EM) Algorithm proposed in \citeA{dempster}. The EM algorithm can  be seen as a clustering method in the sense that it partitions a given dataset into  a predetermined number of subsets. For our purposes, we  call the EM algorithm with 2 clusters and obtain the estimates of the parameters of the mixture distribution. 
%
We also gather statistical evidence that the proposed mixture model 
improves the fit to the S\&P 500 dataset  through Likelihood Ratio Test. 


There are a couple of points to be clarified in the implementation of the EM algorithm. For instance, the likelihood function of a mixture distribution is not concave. Therefore, the EM algorithm may  give some estimates that are only locally optimal. One can initialize the algorithm with random starting points to try to overcome this issue. Another important point, which might be more subtle, is the meaningfulness of the estimates: Suppose that $\hat \mu^2$ and $\hat \mu^1$ represent the mean return vectors for regular (good) and non-regular (bad) market conditions respectively. We would expect that $\hat \mu^2 > \hat \mu^1$. Although the algorithm itself does not guarantee such an outcome, we do actually obtain  $\hat \mu^2 > \hat \mu^1$ and $\hat \rho_2 > \hat \rho_1$, which is in accordance with our intuition (see Table~\ref{table:mix market-wise}). 

\begin{table}[h]\small
\centering
\begin{tabular}{cccccccc}
\hline
                       & \multicolumn{ 2}{c}{Normal} & &\multicolumn{ 4}{c}{Mixture ($\rho_1 = 0.19$, $\rho_2=0.81$)} \\

\cline{2-3}
\cline{5-8}

  Sector         &                     $\mu$ &       ${\sigma}$ &     &   $\mu^1$ &       $\mu^2$  &      ${\sigma^1}$ &    ${\sigma^2}$   \\
\hline

    Energy &                 1.1769 &     6.2823 &     &   -0.0686 &     1.4687 &     8.5162 &     5.5799 \\

Consumer Discretionary             &     1.5112 &     5.3195 &      &   0.4788 &     1.7532 &     8.2673 &     4.3019 \\

Consumer Staples &                 1.3905 &     4.1470 &     &    0.6265 &     1.5696 &     6.7426 &     3.2166 \\

Real Estate &                 1.1514 &     7.2451 &    &     0.3782 &     1.3327 &    11.9825 &     5.5492 \\

Industrials &                 1.2887 &     5.1505 &     &    0.1638 &     1.5523 &     8.1720 &     4.0790 \\

Financials &                1.3322 &     6.2810 &      &   0.7178 &     1.4762 &    10.0766 &     4.9656 \\

Telecommunications Services             &     1.0318 &     5.4780 &     &    0.3681 &     1.1873 &     8.7565 &     4.3408 \\

Information Technology             &     1.7264 &     7.1032 &   &      1.3907 &     1.8051 &    11.8056 &     5.4234 \\

 Materials &                 1.3898 &     5.6955 &      &   0.0673 &     1.6998 &     8.4971 &     4.7437 \\

Health Care &                 1.4164 &     4.6432 &    &     1.3203 &     1.4389 &     6.9005 &     3.9212 \\

 Utilities &                 1.0140 &     4.2801 &     &    0.6794 &     1.0924 &     6.1977 &     3.6764 \\

\hline

\end{tabular}
\caption{Estimates for normal and mixture of normal fits  
with 360-month data (all figures are in percentage). Covariances between the sectors are not reported for brevity. Here, $\sigma^i$  is a vector consisting of the standard deviation of sectors,~$i=~1,2$.}\label{table:mix market-wise}
\end{table}

\section{Risk Minimizing Portfolio Construction}
\label{sec:risk-min}

In this section, we discuss how to obtain risk minimizing portfolios under different probabilistic models  and risk measures. We  start with standard deviation minimization in Section \ref{sec:SDM}, which does not require any distribution information on the returns. Then, we consider CVaR minimization when the sector returns are modeled as multivariate normal distributions in Section \ref{sec:CMN}. These two optimization problems are well-known in the financial engineering literature. Our main contribution is presented in Section \ref{sec:mixture cvar opt}, in which we show how CVaR can be minimized when the sector returns are modeled as a mixture of normals.

{\black
Throughout the paper, we follow the standard convention in the financial industry in which the risk measures VaR and CVaR are defined with respect to the left tail of the return distributions. 
For completeness, we now provide the proper definitions of these risk measures and some basic facts. 

\begin{dfn}\label{dfn:VaR}
Let $\alpha \in (0,1)$. The $\alpha$-level {VaR} of a random variable $Z$ is defined as
\begin{equation*}\label{eq:VaR def general}
{\VaR}_\alpha (Z) = \inf_{z\in\mathbb{R}} \{ \text{P}(z+Z \ge 0) \le 1 - \alpha\}.
\end{equation*}
In other words, ${\VaR}_\alpha (Z) $ is {the negative of the $\alpha$-quantile} of the random variable $Z$.
\end{dfn}

\begin{dfn}\label{dfn:CVaR}
Let $\alpha \in (0,1)$. The $\alpha$-level {CVaR} of a random variable $Z$ is defined as
\begin{equation*}\label{eq:VaR def general}
{\CVaR}_\alpha (Z) = - \text{E}[Z | Z \le  -\text{VaR}_\alpha(Z)] .
\end{equation*}
\end{dfn}
{\black 
\noindent We would like to point out that the random variable $Z$ in Definitions \ref{dfn:VaR} and \ref{dfn:CVaR} represents  the return of a portfolio. One can equivalently define VaR and CVaR with respect to the loss of a portfolio as well.
}

{\black \indent Due to \citeA{rockafellar2002}, ${\CVaR}_\alpha (Z) $ can be computed as the optimal value of the following convex minimization problem
\begin{equation}\label{eq:cvar generic min}
{\CVaR}_\alpha (Z) = \min_{c \in\mathbb{R}}\left\{ c - \frac1\alpha \text{E}[(Z+c)_-]   \right\},
\end{equation}
where $(u)_- := \min\{0,u\}$. We also note that the minimizer of the above optimization problem  gives  ${\VaR}_\alpha (Z) $.
}

The availability of the explicit computation of VaR and CVaR depends on the underlying probability distribution. For instance, for a normal random variable $Z$ with mean~$\nu$ and variance~$\sigma^2$,  the $\alpha$-level VaR and CVaR of $Z$, $\alpha \in(0,1)$, can be computed as 
\[
{\VaR}_\alpha (Z) = -\nu -    \Phi^{-1}(\alpha) \sigma  \ \text{ and } \ 
{\CVaR}_\alpha (Z) = -\nu  + \frac{ \phi( \Phi^{-1}(\alpha) )}{\alpha} \sigma,
\]
where $\phi $ and $\Phi$ are respectively the probability density function (pdf) and  the cumulative distribution function (cdf) of the standard normal distribution. If $\alpha \ge 1$, we will set
$
{\VaR}_\alpha (Z) =- \infty, 
$
for notational convenience.
%
%

}

\subsection{Standard Deviation Minimization}
\label{sec:SDM}

Assuming that the covariance matrix $\Sigma$ is estimated from historical data, we can solve the following problem to minimize the standard deviation of the portfolio return:
\begin{equation}\label{eq:stdev min}
\min_{x \in \Delta_n} \big \{ \sqrt{x^T \Sigma x}  \big \}.
\end{equation}

We note that problem \eqref{eq:stdev min} can be solved efficiently either as  a quadratic program (after squaring the objective function) or as a second-order cone program (SOCP) in a lifted space.

\subsection{CVaR Minimization under Normal Distribution}
\label{sec:CMN}

Let us assume that the vector of sector returns, denoted by $r_N$, is modeled to come from a multivariate normal distribution with mean parameter $\mu$ and covariance matrix $\Sigma$, estimated from historical data.
Then, we can obtain a  portfolio minimizing  CVaR by solving
\begin{equation}\label{eq:cvar normal}
\min_{x \in \Delta_n} \left \{ {\CVaR}_\alpha (r_N^T x)   \right \},
\end{equation}
where the $\alpha$-level CVaR of the return of a portfolio $x$ is computed as
\begin{equation}\label{eq:cvar normal def}
{\CVaR}_\alpha (r_N^T x) = -\mu^T x + \frac{ \phi( \Phi^{-1}(\alpha) )}{\alpha}\sqrt{x^T \Sigma x}.
\end{equation}

We again note that problem \eqref{eq:cvar normal} can be formulated as an SOCP in  a lifted space.

\subsection{CVaR Minimization under Mixture Distribution}
\label{sec:mixture cvar opt}

Now, let us assume that the vector of sector returns, denoted by $r_M$, is modeled to come from a mixture of normal distributions with  parameters $\rho_i$,  $\mu^i$, $\Sigma^i$, $i=1,2$, obtained from the historical data using the technique proposed in Section \ref{sec:two models}.
In this case, we would like to obtain a CVaR minimizing portfolio by solving the following optimization problem:
\begin{equation}\label{eq:cvar mixture}
\min_{x \in \Delta_n} \left \{ {\CVaR}_\alpha (r_M^T x)   \right \}.
\end{equation}
As CVaR is a convex function \cite{pflug2000},  the optimization problem \eqref{eq:cvar mixture} is again a convex program. 
{\black Since the CVaR of a mixture distribution does not have a closed form expression, we will utilize the expression \eqref{eq:cvar generic min} to obtain an explicit convex program which can be used to the CVaR minimization problem  \eqref{eq:cvar mixture}.} 
For notational purposes, let us define  $\nu_i := {\mu^i}^Tx$, $\sigma_i^2 := x^T\Sigma^i x$, $i=1,2$ and $V:=\VaR_\alpha (r_M^T x)$.

\subsubsection{{\black Computing and Optimizing}  CVaR under Mixture Distribution}
\label{sec:mixture cvar opt comp}
We first note that ${\CVaR}_\alpha (r_M^T x)$ can be computed analytically if $V$ is at hand (also derived in \citeA{Broda2011}) through
\begin{equation}\label{eq:cvar mixture analytic}
\begin{split}
{\CVaR}_\alpha(r_M^Tx) =&-\frac1\alpha \int_{-\infty}^{-V} y \sum_{i=1}^2 \rho_i  \frac{1}{\sqrt{2\pi\sigma_i^2}}  e^{ -\frac{(y-\nu_i)^2}{2\sigma_i^2} } dy \\
=&-\frac1\alpha\sum_{i=1}^2 \rho_i    \int_{-\infty}^{-V} y  \frac{1}{\sqrt{2\pi\sigma_i^2}}  e^{ -\frac{(y-\nu_i)^2}{2\sigma_i^2} } dy \\
=&\frac1\alpha\sum_{i=1}^2  \rho_i [ \sigma_i^2 {\phi(\nu_i, \sigma_i^2, -V)}- \nu_i { \Phi(\nu_i, \sigma_i^2, -V)}] .
\end{split}
\end{equation}
Here, $\phi(\nu, \sigma^2, y)$ and $\Phi(\nu, \sigma^2, y)$ are respectively the pdf and cdf of the normal distribution with mean  $\nu$ and variance $\sigma^2$ evaluated at the point $y$.

However, $V=\VaR_\alpha (r_M^T x)$ does not have a closed form expression either. 
{\black Fortunately, one does not need to have the exact value of $V$ to solve the problem \eqref{eq:cvar mixture}. Adapting the generic definition 
{\black of} CVaR in equation \eqref{eq:cvar generic min} to the special case of mixture of normals, we obtain
\begin{equation}\label{eq:cvar mixture comp opt}
{\CVaR}_\alpha(r_M^Tx) = \min_{c \in \mathbb{R}} \left\{c  +  \frac1\alpha\sum_{i=1}^2  \rho_i [ \sigma_i^2 {\phi(\nu_i, \sigma_i^2, -c)}- (c+\nu_i) { \Phi(\nu_i, \sigma_i^2, -c)}] \right \},
\end{equation}
where the optimal solution $c^*$ corresponds to $\VaR_\alpha (r_M^T x)$.
Finally, the resulting CVaR minimization  problem can be stated explicitly as follows:
\begin{equation}\label{eq:cvar mixture one-shot}
\min_{ (x,c) \in \Delta_n \times \mathbb{R}  } \left\{c  +  \frac1\alpha\sum_{i=1}^2  \rho_i \left [ (x^T\Sigma_i x) {\phi({\mu^i}^Tx,  x^T\Sigma^i x, -c)}- (c+{\mu^i}^Tx) { \Phi({\mu^i}^Tx,  x^T\Sigma^i x, -c)} \right]   \right \}.
\end{equation}
}

\subsubsection{Approximating  CVaR under Mixture Distribution}
\label{sec:approx CVaR}

%
%

In the previous section, {\black we mentioned that ${\CVaR}_\alpha(r_M^Tx)$ can be  computed by solving the convex program \eqref{eq:cvar mixture comp opt}. In this section, we develop an explicit and second-order cone representable over-approximation of the same quantity, for the reasons that will become clearer  in Section \ref{sec:mixture dist} when we incorporate the Black-Litterman views into the portfolio construction procedure via inverse optimization. We also provide the approximation guarantee of the proposed approximation together with some empirical evidence of its accuracy when applied to the S\&P 500 dataset.
}


Below, we first provide  under and over-approximations of the function $\VaR_\alpha(r^Tx)$, {\black which will be the key in the approximation of the CVaR function}.
\begin{prop}\label{prop:VaR lower bound}
{\black Let $\alpha \in (0, 1)$.} Then,
$\max_{i=1,2} \{ {\VaR}_{\alpha/\rho_i}(r_{M,i}^Tx) \} \le \VaR_\alpha(r_M^Tx)$.
\end{prop}
\begin{proof}
For any $k \in \{1,2\}$, we have
\begin{equation*}
\begin{split}
\text{P} \left(r_M^Tx \le -   {\VaR}_{\alpha/\rho_k}(r_{M,k}^Tx)  \right)
=& \sum_{j=1}^2\rho_j \text{P} \left(r_{M,j}^Tx \le -  {\VaR}_{\alpha/\rho_k}(r_{M,k}^Tx)  \right) \\
\ge & \rho_k \text{P} \left(r_{M,k}^Tx \le - {\VaR}_{\alpha/\rho_k}(r_{M,k}^Tx) \right) \\
=& \alpha.
\end{split}
\end{equation*}
This implies that ${\VaR}_{\alpha/\rho_i}(r_{M,i}^Tx) \le {\VaR}_\alpha(r_M^Tx)$, hence the result follows.
\end{proof}
{\noindent We note that Proposition \ref{prop:VaR lower bound} is guaranteed to provide a non-trivial lower bound on $\VaR_\alpha(r_M^Tx)$ for $\alpha \in (0, 1/2)$, which is a very loose assumption.}

{\black Let us define the set
$
\tilde \Delta_m :=  \{\theta \in\mathbb{R}^m: \ \sum_{i=1}^m \theta_i = 1, \ \theta > 0  \}.
$ 
}
\begin{prop}\label{prop:VaR upper bound}
Let {\black  $\alpha \in (0, 1)$}, $\alpha \theta_i \le \rho_i$, $i=1,2$ and $\theta \in {\black \tilde \Delta_2}$. Then,
${\VaR}_\alpha(r_M^Tx) \le \max_{i=1,2} \{ {\VaR}_{\alpha \theta_i /\rho_i}(r_{M,i}^Tx) \} $.
\end{prop}
\begin{proof}
The result follows since we have
\begin{equation*}
\begin{split}
\text{P} \left(r_M^Tx \le - \max_{i=1,2} \{ {\VaR}_{\alpha \theta_i /\rho_i}(r_{M,i}^Tx) \} \right)
=& \sum_{j=1}^2\rho_j \text{P} \left(r_{M,j}^Tx \le - \max_{i=1,2} \{ {\VaR}_{\alpha  \theta_i /\rho_i}(r_{M,i}^Tx) \} \right) \\
\le& \sum_{j=1}^2\rho_j \text{P} \left(r_{M,j}^Tx \le -  {\VaR}_{\alpha  \theta_j /\rho_j}(r_{M,j}^Tx)  \right) \\
=& \sum_{j=1}^2\rho_j \alpha  \theta_j /\rho_j  \\
=& \alpha.
\end{split}
\end{equation*}
\end{proof} 
{\black \noindent We note that by restricting $\theta_i$ to be positive for each $i=1,2$, we eliminate the possibility of obtaining a trivial upper bound on $\VaR_\alpha(r_M^Tx)$ in Proposition \ref{prop:VaR upper bound}.}


{\black We now  give  an over-approximation of the function ${\CVaR}_\alpha(r^Tx)$.}
\begin{prop} \label{prop:cvar approx}
{\black Let $\alpha \in (0, \min_{i=1,2}\{\rho_i\})$}. Then,
${\CVaR}_\alpha(r_M^Tx) \le \sum_{i=1}^2 {\CVaR}_{\alpha /\rho_i}(r_{M,i}^Tx) $.
\end{prop}
\begin{proof}
For notational purposes, let us define
$V_i :=\VaR_{\alpha/\rho_i} (r_{M,i}^T x)$. Note that $V \ge V_i$ for all $i=1,2$ due to Proposition~\ref{prop:VaR lower bound}. Let $f_i$ denote the pdf of the random variable $r_{M,i}^T x$.
Then, we have
\begin{subequations}
\begin{align*}
{{\CVaR}}_\alpha(r_M^Tx) =&-\frac1\alpha \int_{-\infty}^{-V} y \sum_{i=1}^2 \rho_i  f_i(y)dy 
=-\sum_{i=1}^2\frac1{\alpha/\rho_i} \int_{-\infty}^{-V} y  f_i(y)dy \\
\le&-\sum_{i=1}^2\frac1{\alpha/\rho_i} \int_{-\infty}^{-V_i} y  f_i(y)dy 
= \sum_{i=1}^2 {\CVaR}_{\alpha /\rho_i}(r_{M,i}^Tx) ,
\end{align*}
\end{subequations}
{\black where the inequality follows due to the fact that $V \ge V_i \ge 0$ for all $i=1,2$ and the integrals involved are nonpositive.}
\end{proof}
{\black 
\noindent We note that the requirement in Proposition \ref{prop:cvar approx} can be satisfied for risk averse investors by selecting a small enough $\alpha$. However, this condition can be restrictive in  cases when the decision maker is not very risk averse.
}

{ 
We also derive an explicit  under-approximation of ${\CVaR}_\alpha(r_M^Tx)$.
{\black
\begin{prop} \label{prop:cvar approx lower}
Let {$\alpha \in (0, 1)$}. Then,
$ \max_{i=1,2} \{{\CVaR}_{\alpha/\rho_i}(r_{M,i}^Tx)\}  \le {\CVaR}_\alpha(r_M^Tx) $.
\end{prop}
\begin{proof}
Applying the CVaR identity   \eqref{eq:cvar generic min} to the mixture random variable $r_M^Tx$ at the level $\alpha$, we have
\begin{subequations}
\begin{align*}
{\CVaR}_\alpha(r_M^Tx) =& \min_{c \in\mathbb{R}}\left\{ c - \frac1\alpha \text{E}[(r_M^Tx+c)_-]   \right\} \\
=& \min_{c \in\mathbb{R}}\left\{ c - \frac1\alpha \left( \rho_1 \text{E}[(r_{M,1}^Tx+c)_-] + \rho_2 \text{E}[(r_{M,2}^Tx+c)_-] \right)   \right\} \\
\ge& \min_{c \in\mathbb{R}}\left\{ c - \frac1{\alpha/\rho_1} \text{E}[(r_{M,1}^Tx+c)_-] \right \} + \min_{c \in\mathbb{R}}\left\{ -\frac{1}{\alpha/\rho_2} \text{E}[(r_{M,2}^Tx+c)_-]   \right\} \\
\ge& {\CVaR}_{\alpha/\rho_1}(r_{M,1}^Tx).
\end{align*}
\end{subequations}
Here, the first inequality follows from the fact that the minimum of the sum of two functions is at least the sum of the minimum of these functions while the second inequality follows from the definition of ${\CVaR}_{\alpha/\rho_1}(r_{M,1}^Tx)$ and the fact that $\text{E}[(r_{M,2}^Tx+c)_-] \le 0$.

Since the same argument can be repeated for $ {\CVaR}_{\alpha/\rho_2}(r_{M,2}^Tx)$, the statement of the proposition follows.
%
%
\end{proof}
}

%

{\black We now give a guarantee for the  approximation proposed   in Proposition \ref{prop:cvar approx}.}
\begin{prop} \label{prop:cvar approx guarantee}
{\black Let $\alpha \in (0, \min_{i=1,2}\{\rho_i\})$.} Then,
$
{\CVaR}_\alpha(r_M^Tx)
\le 
\sum_{i=1}^2 {\CVaR}_{\alpha /\rho_i}(r_{M,i}^Tx) 
\le \kappa {\CVaR}_\alpha(r_M^Tx),
$
where $\kappa = \frac{\sum_{i=1}^2 {\CVaR}_{\alpha /\rho_i}(r_{M,i}^Tx)}{\max_{i=1,2} \{{\CVaR}_{\alpha /\rho_i}(r_{M,i}^Tx)\}} \le 2 $.
\end{prop}
\begin{proof}
The desired result follows {\black from Proposition \ref{prop:cvar approx lower}}  since
\begin{subequations}
\begin{align*}
\sum_{i=1}^2 {\CVaR}_{\alpha /\rho_i}(r_{M,i}^Tx) = & \left ( \frac{\sum_{i=1}^2 {\CVaR}_{\alpha /\rho_i}(r_{M,i}^Tx)}{\max_{i=1,2} \{{\CVaR}_{\alpha /\rho_i}(r_{M,i}^Tx)\}}  \right)  \left ( \max_{i=1,2} \{{\CVaR}_{\alpha /\rho_i}(r_{M,i}^Tx)\} \right ) \\
\le &  \kappa {\CVaR}_\alpha(r_M^Tx),
\end{align*}
\end{subequations}
where $\kappa \le 2$ by definition.
\end{proof}

We remark that Propositions \ref{prop:VaR lower bound}--{\black\ref{prop:cvar approx guarantee}} can be extended in a straightforward manner to the mixture of  $m$ random variables,  $m \ge 2$, that are not necessarily normal. {In particular, Proposition \ref{prop:VaR lower bound}  yields a non-trivial bound when $\alpha \in (0,1/m)$,  Proposition \ref{prop:cvar approx} remains valid if  $\alpha \in (0,  \min_{i=1,\dots,m}\{\rho_i\})$}, {\black and the approximation factor $\kappa$ in Proposition \ref{prop:cvar approx guarantee} satisfies the condition $\kappa \le m$.}

We note that the over-approximation in Proposition \ref{prop:cvar approx} has an approximation factor of $2$ in the worst case. In practical settings related to the application in this paper, the accuracy of this approximation is  higher as the contribution to the CVaR of the mixture distribution due to one of the normals (in particular, the one on the ``left'') is  larger than the other.

We also point out  that the approximation guarantee proven in Proposition \ref{prop:cvar approx guarantee}  is valid for \textit{any} portfolio vector $x$. Since we are primarily interested in the performance comparison of the optimal solutions obtained from the portfolio optimization problems with the exact and approximate CVaR objectives, we designed the following experiment to test the empirical behavior of the proposed approximation:
Let $x^*$ be {\black an} optimal solution of the exact CVaR minimization problem \eqref{eq:cvar mixture one-shot}, {\black which is  obtained by using the interior point solver IPOPT \cite{wachter}}. Also, let   $x'$ be {\black an} optimal solution of the approximate CVaR minimization problem
\begin{equation}
\min_{x \in \Delta_n} \left \{ \sum_{i=1}^2 {\CVaR}_{\alpha /\rho_i}(r_{M,i}^Tx)   \right \}, \label{eq:cvar mixture approx opt}
\end{equation}
{\black which is obtained by using the conic interior point solver MOSEK \cite{MOSEK} after transforming this optimization problem into an explicit SOCP.}

We use the 360-month S\&P dataset mentioned in Section \ref{sec:data coll} for this illustration. For each month $t=181, \dots, 360$, we estimate the parameters of the mixture distribution using  the historical return vectors in the interval $[t-H, t-1]$, where $H$ is a rolling horizon window. 
We compute the  ``percentage error" of each of the 180 pairs of solutions as
\[
e= 100 \times \left ( \frac{ {\CVaR}_\alpha(r_M^T x')}{ {\CVaR}_\alpha(r_M^T  x^*)} - 1 \right ).
\]
Finally, we report the arithmetic and geometric average for three different $H$ values, namely 60, 120 and 180,  in Table  \ref{tab:cvar approx}. We can see that the approximation is quite accurate, especially when $H$ is chosen as 120 or 180.
\begin{table}[h]
\centering
\begin{tabular}{|c|c|c|c|}
\hline
   Average &       $H=60$  &    $  H=120$  &     $ H=180 $ \\
\hline
Arithmetic &       2.40 &       0.45 &       0.76 \\
\hline
Geometric &       0.75 &       0.17 &       0.42 \\
\hline
\end{tabular}  
\caption{Average percentage error of the CVaR approximation.}\label{tab:cvar approx}
\end{table}

In Section \ref{sec:mixture dist}, we use the approximate CVaR minimization problem \eqref{eq:cvar mixture approx opt} in the adaptation of the Black-Litterman approach for the CVaR minimization under the mixture distribution.
}

\section{Incorporating Market Information into Risk Minimizing Portfolios}
\label{sec:combine}

In the previous section, we presented ways to obtain risk minimizing portfolios under different probabilistic models and risk measures. In this section, our  objective is to combine market information with these risk minimizing portfolios, which will be achieved through a BL-type approach. After reviewing the basic construction of the BL approach through Theil's mixed estimation procedure, we also summarize a modern interpretation via inverse optimization in Section \ref{sec:BL review}. Then, we extend the inverse optimization model from MV optimization to CVaR minimization under both normal and mixture distributions in Section \ref{sec:extend BL CVaR}.

\subsection{Review of the Black-Litterman Model}
\label{sec:BL review}
The classical BL model is driven by two factors: market equilibrium and investor views. Let us first discuss how market equilibrium is obtained through ``reverse optimization". Consider the MV optimization problem with $\mathcal{X} = \mathbb{R}^n$: 
\begin{equation}\label{eq:MMV}
\max_{x \in \mathbb{R}^n} \   \mu^T x - \delta x^T \Sigma x.
\end{equation}
Assuming that $\Sigma$ is estimated from the historical data as $\hat \Sigma$, $\delta$ is predetermined and the percentage market capitalization vector is given as $ x^\text{m}$, we can compute the ``implied returns" which induce this market as
\begin{equation}\label{eq:CAPM}
\Pi := 2\delta \hat \Sigma x^\text{m},
\end{equation}
by writing down the optimality conditions for the unconstrained optimization problem \eqref{eq:MMV}. This derivation is the basis of the capital asset pricing
model (CAPM).  In the absence of  any other views, an investor should invest proportional to $x^\text{m}$.

The second force that drives the portfolio away from the market equilibrium is the investor views. The BL model incorporates these views in the portfolio allocation  via a certain ``mixing" procedure. The basic construction is as follows: Suppose that the random return vector is distributed normally with mean $\mu$ and covariance matrix $\hat \Sigma$, that is,
$
r \sim \text{N}(\mu, \hat \Sigma)
$.
However, according to this model, $\mu$ itself is random with
\begin{equation}\label{eq:r random}
\mu \sim \text{N}(\Pi, \tau\hat \Sigma),
\end{equation}
where $\tau$ is chosen exogenously, and a smaller value implies strong confidence in the market equilibrium. The investor views are expressed with
linear equations of the form $P\mu = q$ with certain confidence level quantified by the covariance matrix $\Omega$. To be more precise, the BL model assumes that
\begin{equation}\label{eq:view random}
P\mu \sim \text{N}(q, \Omega),
\end{equation}
where $\Omega$ is a diagonal matrix (so that views are independent) with positive diagonal entries. Again, smaller values indicate stronger confidence in the views.

At this point, we remind the reader the solution of the generalized least squares problem, which is a key to obtain the BL estimates:
\begin{prop}\label{prop:gls}\cite{aitken1936}
Consider the generalized least squares problem $Ax = b + \epsilon$, where $\epsilon \sim \text{N}(0, \Omega)$ with $\Omega \succ 0$. Then,
\begin{equation*}\label{eq:solution GLS}
\hat x := (A^T\Omega^{-1}A)^{-1} A^T\Omega^{-1}b,
\end{equation*}
minimizes  the $|| \cdot ||_{ \Omega} $ norm of the error, where $|| v ||_{ \Omega} := \sqrt{v^T  \Omega^{-1} v} $.
\end{prop}

In order to obtain an estimate for $\mu$, one can use  Theil's Mixed Estimation  \cite{theil1961, theil1963} procedure, which is just a corollary of Proposition  \ref{prop:gls}.
 In particular, we rewrite equations \eqref{eq:r random}--\eqref{eq:view random} as
\begin{equation}\label{eq:mixedBL}
\begin{bmatrix} I \\ P \end{bmatrix} \mu = \begin{bmatrix} \Pi \\ q \end{bmatrix} + \begin{bmatrix}  \epsilon_1 \\  \epsilon_2 \end{bmatrix}, \ \
\text{where} \ \
\epsilon_1 \sim \text{N}(0, \tau \hat \Sigma), \ \epsilon_2 \sim \text{N}(0,  \Omega).
\end{equation}
The  BL estimate for the mean return vector $\mu$ is obtained by
\begin{equation}\label{eq:mu BL}
\mu^\text{BL} = ((\tau \hat \Sigma)^{-1}+ P^T  \Omega^{-1}P)^{-1} ((\tau \hat \Sigma)^{-1} \Pi+P^T \Omega^{-1}q).
\end{equation}

Alternative derivations for $\mu^\text{BL} $ and $\Sigma^\text{BL}$ can be seen in \citeA{he1999, satchell2000} based on a Bayesian interpretation and in \citeA{bertsimas2012} via inverse optimization. Here, we will outline the latter approach since it has motivated the extension presented in the next section.
Recall that the optimality condition of the MV problem \eqref{eq:MMV}  can be written as
$
\mu = 2\delta \Sigma x^*
$.
In the inverse optimization problem, given a solution $x^*$, we search for the parameters that make this solution optimal. Assuming that $x^*=x^\text{m}$, $\delta$ is given and $\Sigma \succeq 0$, the solution of the inverse problem satisfies the following linear equation in $\mu$ and $\Sigma$:
\begin{equation}\label{eq:inv mkt}
\mu = 2\delta \Sigma x^\text{m}.
\end{equation}
This is the market equilibrium in the BL model. Similarly, investor views are expressed as linear equations in $\mu$ as
\begin{equation}\label{eq:inv view}
P\mu=q.
\end{equation}
In this approach, the aim is to solve the system \eqref{eq:inv mkt}-\eqref{eq:inv view} and $\Sigma \succeq 0$ simultaneously to obtain estimates for $\mu$ and $\Sigma$ in a certain least squares sense. In particular, one solves the following norm minimization problem:
\begin{equation}\label{eq:inv}
\min_{\mu, \Sigma \succeq 0} \begin{Vmatrix} \mu -2 \delta \Sigma x^\text{m} \\  P\mu - q \end{Vmatrix}
\end{equation}
The flexibility of this approach allows us to use different norms, which can lead to tractable conic programs. For instance, if the underlying norm is $|| v ||_{\bar \Omega} $ where $\bar \Omega =~  \begin{bmatrix} \tau \hat \Sigma & 0 \\ 0 & \Omega \end{bmatrix}$ and $\Sigma = \hat \Sigma$, then the optimal solution of the problem \eqref{eq:inv} is precisely $\mu^\text{BL}$ in \eqref{eq:mu BL}. Hence, the BL estimate is just a special case of the inverse optimization approach. Although this derivation allows one to specify views on  the variance estimate $\Sigma$ through, for instance, factor models, we will use historical variance estimates directly not to over-complicate our analysis.


\subsection{Extension of the Black-Litterman Model to CVaR Minimization Problems}
\label{sec:extend BL CVaR}
In this section, we extend the BL Model from MV optimization to CVaR minimization problems under both normal and mixture distributions. We use the inverse optimization framework in \citeA{bertsimas2012} outlined in the previous section, and also present an equivalent interpretation  in terms of  Theil's Mixed Estimation principle.

\subsubsection{Normal Distribution}

Let us consider the CVaR minimization problem under a normal distribution defined in \eqref{eq:cvar normal}, which can be explicitly stated as
\begin{equation}\label{eq:cvar normal exp}
\min \big \{  -\mu^T x + z \sqrt{x^T \Sigma x} :\ e^T x = 1, x \ge0  \big \},
\end{equation}
where $e$ is the vector of ones and $z := \frac{\phi\left( \Phi^{-1}(\alpha) \right)}{\alpha}$.
Let us associate dual variables $\lambda$ and $\gamma$ to the equality and inequality constraints in problem \eqref{eq:cvar normal exp}. Noting that the problem is a convex program,  first-order necessary and sufficient conditions are:
\begin{equation} \label{eq:cvar normal exp opt cond}
 -\mu + z \frac{\Sigma x}{\sqrt{x^T \Sigma x}}   - \lambda e - \gamma = 0,
\ e^T x = 1, \ x \ge 0 , \
\ \gamma_j x_j = 0, \ \gamma \ge 0.
\end{equation}

Let us now consider the inverse problem, where we will treat a market portfolio $x^{\text{m}}$ as an optimal primal solution to problem \eqref{eq:cvar normal exp}, and we will search for the parameters $\mu$, $\Sigma$ and dual variables $\lambda$, $\gamma$ that satisfy the optimality conditions in \eqref{eq:cvar normal exp opt cond}.
At this point, we will assume that the covariance matrix $\Sigma$ is estimated by the sample covariance matrix $\hat \Sigma$. There are two main reasons for this assumption: First, it is generally accepted in the financial engineering literature that the estimation of the covariance matrix is reasonably  accurate as opposed to the mean estimation. Second, this assumption allows us to have an inverse problem that is linear in the remaining unknowns.
A further simplification can be made by taking into account the fact that each sector is invested a positive amount in the market portfolio since the percentage market capitalization is positive for each sector, i.e.,  $x^{\text{m}} > 0$. Hence, we have that $\gamma = 0$. Therefore, the inverse problem is given below, where $\mu$ and $\lambda$ are the only remaining unknowns:
\begin{equation} \label{eq:cvar normal inverse}
\mu + \lambda e = \tilde \mu_N := z \frac{\hat \Sigma x^{\text{m}}}{\sqrt{{x^{\text{m}}}^T \hat  \Sigma x^{\text{m}}}} .
\end{equation}
We will call  \eqref{eq:cvar normal inverse} the market equilibrium equation under normal distribution.

In our approach, the investor views are expressed simply as
\begin{equation} \label{eq:cvar normal view}
\mu = \hat \mu ,
\end{equation}
where $\hat \mu$ is the sample average estimated from the data. We then solve the following optimization problem to obtain adjusted $\mu$ estimates
\begin{equation}\label{eq:cvar normal norm}
\min_{\mu, \lambda} \begin{Vmatrix} \mu + \lambda e - \tilde \mu_N \\  \mu - \hat \mu \end{Vmatrix}_{\hat\Omega},
\end{equation}
where ${\hat\Omega} = \begin{bmatrix} \tau \hat \Sigma & 0 \\ 0 & \hat \Sigma \end{bmatrix}$, with $\tau > 0$. Note that as $\tau$ goes to zero, we put more emphasis on the market equilibrium and investor views become less important. In the other extreme, as $\tau$ gets larger, then the investor views become dominant and we recover the sample estimate. We would like to point out that our model allows more general investor views than the one used above. For instance, similar to the original BL model, linear equations in terms of the mean return vector given as $P\mu=q$ can be incorporated into our model as well. {\black In this case, the following optimization problem can be solved to find adjusted $\mu$ estimates
\begin{equation*}
\min_{\mu, \lambda} \begin{Vmatrix} \mu + \lambda e - \tilde \mu_N \\  P\mu - q \end{Vmatrix}_{\bar\Omega},
\end{equation*}
where ${\bar\Omega} = \begin{bmatrix} \tau \hat \Sigma & 0 \\ 0 & \Omega \end{bmatrix}$ and the confidence of investor views is expressed as in \eqref{eq:view random}. 
}

Note that problem \eqref{eq:cvar normal norm} has a closed form solution. Below, we also provide an equivalent interpretation in terms of Theil's Mixed Estimation principle.
\begin{prop}
The solution of   \eqref{eq:cvar normal norm} is equivalent to the solution of the following mixed estimation problem:
\begin{equation*}\label{eq:mixedBL cvar normal}
\begin{bmatrix} I & e \\ I & 0 \end{bmatrix} \begin{bmatrix} \mu \\  \lambda\end{bmatrix}= \begin{bmatrix} \tilde \mu_N \\ \hat \mu \end{bmatrix} + \begin{bmatrix}  \epsilon_1 \\  \epsilon_2 \end{bmatrix}, \ \
\text{where} \ \
\epsilon_1 \sim \text{N}(0, \tau \hat \Sigma), \ \epsilon_2 \sim \text{N}(0, \hat \Sigma).
\end{equation*}
\end{prop}

\subsubsection{Mixture Distribution}
\label{sec:mixture dist}

Let us now consider the CVaR minimization problem under the mixture distribution defined in~\eqref{eq:cvar mixture one-shot}. 
{\black Our aim is again to analyze the corresponding inverse optimization problem, which might lead us to the BL extension under the mixture model. Unfortunately, the optimality conditions of this minimization problem yield highly nonlinear  expressions in terms of the  mean return parameters $\mu_1$ and $ \mu_2$. Therefore, it seems  very unlikely that we would be able to solve the related inverse optimization problem in closed form. 
We remark that this has been the main motivation behind the CVaR approximation developed in Section~\ref{sec:approx CVaR}.}
{Due to Proposition \ref{prop:cvar approx}, we have an explicit expression that over-approximates $ {\CVaR}_\alpha (r_M^T x)$, whose accuracy is demonstrated theoretically by Proposition \ref{prop:cvar approx guarantee} and empirically by the experiments summarized in Table \ref{tab:cvar approx}. Therefore, we will base our construction on the approximate CVaR minimization problem \eqref{eq:cvar mixture approx opt}.}
To be more precise, let us consider the following convex program
\begin{equation} \label{eq:cvar mixture exp}
\min \left \{  \sum_{i=1}^2  \left ( -{\mu^i}^T x + z_i {\black (\rho_i)} \sqrt{x^T \Sigma^i x} \right ) :  e^T x = 1, x \ge0 \right\},
\end{equation}
where $z_i {\black (\rho_i)} := \frac{\phi\left( \Phi^{-1}(\alpha/\rho_i) \right)}{\alpha/\rho_i}$.
Let us again associate dual variables $\lambda$ and $\gamma$ to the equality and inequality constraints, and write down the  first-order necessary and sufficient conditions as:
\begin{equation} \label{eq:cvar mixture exp opt cond}
\sum_{i=1}^2 \left( - \mu^i + z_i {\black (\rho_i)} \frac{\Sigma^i x}{\sqrt{x^T \Sigma^i x}} \right ) - \lambda e - \gamma = 0,
\ e^T x = 1, \ x \ge 0 , \
\ \gamma_j x_j = 0, \ \gamma \ge 0.
\end{equation}

Using the inverse optimization framework, we will again treat  a market portfolio $x^{\text{m}}$  as   an optimal primal solution to problem \eqref{eq:cvar mixture exp}, and we will search for the parameters $\mu^i$, $\Sigma^i$, {\black $\rho_i$}, $i=1,2$, and dual variables $\lambda$, $\gamma$ that satisfy the optimality conditions \eqref{eq:cvar mixture exp opt cond}.
{\black Since our preliminary experiments have demonstrated that the estimation of $\rho_i$'s is quite stable across different data subsets, we  make the assumption that  in addition to the covariance matrices $\Sigma^i$'s, the probabilities $\rho_i$'s  are also estimated from the historical data, using the technique developed in Section \ref{sec:two models}. We will denote these estimated probabilities as $\hat{\rho}_i$, $i=1,2$.}  
By the same reasoning as above, we have that $\gamma = 0$. Finally, the inverse problem is given below, where $\mu^i$, $i=1,2$, and $\lambda$ are the  remaining unknowns:
\begin{equation} \label{eq:cvar mixture inverse}
\sum_{i=1}^2 \mu^i + \lambda e = \tilde \mu_M := \sum_{i=1}^2 z_i {\black (\hat \rho_i)} \frac{\hat \Sigma^i  x^{\text{m}}}{\sqrt{{x^{\text{m}}}^T \hat  \Sigma^i x^{\text{m}}}} .
\end{equation}
We will call  \eqref{eq:cvar mixture inverse} the market equilibrium equation under a mixture distribution.

Similarly, the investor views are expressed  as
\begin{equation} \label{eq:cvar mixture view}
\mu^i = \hat \mu^i, \ i=1,2 ,
\end{equation}
where $\hat \mu^i$'s  are estimated from the historical data. 
We then solve the following optimization problem to obtain adjusted $\mu^i$ estimates
\begin{equation}\label{eq:cvar mixture norm}
\min_{\mu, \lambda} \begin{Vmatrix} \sum_{i=1}^2 \mu^i + \lambda e - \tilde \mu_M \\  \mu^1 - \hat \mu^1 \\  \mu^2 - \hat \mu^2  \end{Vmatrix}_{\hat\Omega},
\end{equation}
where ${\hat\Omega} = \begin{bmatrix} \tau \hat \Sigma & 0 & 0\\ 0 & \hat \Sigma^1 & 0 \\ 0 &0 & \hat \Sigma^2 \end{bmatrix}$, with $\tau > 0$.

We point out a difference between the two models under normal and mixture distributions. In the former case, as $\tau$ approaches zero, we recover the market portfolio. However, in the latter, we do not necessarily recover the market portfolio  due to added flexibility in the mixture model. On the other hand, as $\tau$ gets larger, the investor views become more important and we recover the estimates $\mu^i$, as before.

We provide an equivalent interpretation of the problem \eqref{eq:cvar mixture norm} below.
\begin{prop}
Solution of   \eqref{eq:cvar mixture norm} is equivalent to the solution of the following mixed estimation problem:
\begin{equation*}\label{eq:mixedBL cvar normal}
\begin{bmatrix} I & I & e \\ I &  0& 0 \\ 0 & I & 0 \end{bmatrix} \begin{bmatrix} \mu^1 \\ \mu^2 \\  \lambda\end{bmatrix}= \begin{bmatrix} \tilde \mu_M \\ \hat \mu^1 \\ \hat \mu^2 \end{bmatrix} + \begin{bmatrix}  \epsilon_1 \\  \epsilon_2 \\ \epsilon_3 \end{bmatrix}, \ \
\text{where} \ \
\epsilon_1 \sim \text{N}(0, \tau \hat \Sigma), \ \epsilon_2 \sim \text{N}(0, \hat \Sigma^1), \ \epsilon_3 \sim \text{N}(0, \hat \Sigma^2).
\end{equation*}
\end{prop}

\section{Computational Results}
\label{sec:comp}

In this section, we present the results of our computational experiments on the S\&P 500 dataset. We explain the experimental setting in Section \ref{sec:exp set}, including a proper definition of two types of market-based portfolios, our criterion for preferable portfolios and the rolling horizon based back-testing. In Section \ref{sec: risk min vs mkt}, we provide a comparison of these two market portfolios with three other risk minimizing portfolios obtained through the analytical approaches outlined in Section \ref{sec:risk-min}.  In Section \ref{sec:mkt vs BL}, we  apply the BL-approach proposed in Section \ref{sec:extend BL CVaR} for CVaR minimization and show that the resulting portfolios may dominate the market-based approach in terms of both expected return and risk measures. 
Finally, in Section \ref{sec:replication}, we conduct extensive simulations to replicate the results on a synthetic dataset.

\subsection{Experimental Setting}
\label{sec:exp set}

Let us start by explaining the details of our experimental setting.
We use the 360-month S\&P dataset mentioned in Section \ref{sec:data coll}. Five portfolios, two market-based and three risk minimizing, are constructed for each month $t=181, \dots, 360$, using only the historical return and market capitalization data in the interval $[t-H, t-1]$, where $H$ is a rolling horizon window. This procedure gives us 180 portfolios for each strategy. For a fixed strategy $S$, these portfolios, denoted by $x_S^t$, are evaluated using the return information at $R^{t}$, by simply computing $p_S^t := {R^t}^T x_S^t$. The performance of a portfolio construction strategy is evaluated  with respect to average reward, standard deviation and 1\% CVaR measures by using the dataset $\{p_S^t: \ t=181, \dots, 360\}$. {We also report two risk-adjusted performance measures: the ratio of average reward to standard deviation (also known as the Sharpe ratio) and 1\% CVaR.}

We now discuss how these five portfolios are obtained. We construct two market-based portfolios for $t=181, \dots, 360$ using the market capitalization information $M^t$:
\begin{itemize}
\item Last Market Portfolio ($\mathsf{LstM}$): We simply choose $x_{\text{LstM}}^t = M^{t-1}$.
\item Average Market Portfolio ($\mathsf{AvgM}$): We compute $x_{\text{AvgM}}^t = \frac{1}{H}\sum_{t'=t-H}^{t-1} M^{t'}$.
\end{itemize}
We construct  three risk minimizing portfolios for $t=181, \dots, 360$  using the returns $R^t$:
\begin{itemize}
\item Standard Deviation Minimizing Portfolio ($\mathsf{StDev}$): We obtain $x_{\text{StDev}}^t $ by solving \eqref{eq:stdev min}, where $\Sigma$ is estimated from the data.
\item CVaR Minimizing Portfolio under Normal Distribution ($\mathsf{CVaR\_N}$): We obtain $x_{\text{CVaR\_N}}^t $ by solving \eqref{eq:cvar normal}  with $\alpha=0.01$, where $\mu$ and $\Sigma$ are estimated from the data.
\item CVaR Minimizing Portfolio under Mixture Distribution   ($\mathsf{CVaR\_M}$): We obtain $x_{\text{CVaR\_M}}^t $ by solving \eqref{eq:cvar mixture one-shot} with $\alpha=0.01$, where $\rho_i$, $\mu^i$ and $\Sigma^i$ are estimated from the data using the EM algorithm.
\end{itemize}
We use MOSEK \cite{MOSEK} to solve problems  \eqref{eq:stdev min} and  \eqref{eq:cvar normal}, and IPOPT \cite{wachter} to solve problem \eqref{eq:cvar mixture one-shot}.


%
%
%
%

\vspace{-0mm}
\subsection{Market-Based  vs.  Risk Minimizing Portfolios}
\label{sec: risk min vs mkt}

We first compare the performance of two market-based and {\black three} risk minimizing portfolios with respect to different $H$ values in Figure \ref{fig:market vs risk-min}. We have several observations:  First and foremost, market-based approaches result in higher reward, higher risk portfolios compared to risk minimizing portfolios. An interesting observation is that $\mathsf{AvgM}$ performs better with increasing values of $H$ in terms of both average return and risk measures, and that it has a higher average return than $\mathsf{LstM}$ with typically slightly larger risk (except for CVaR when $H=180$). This suggests that historical data can be useful even in market-based portfolio construction. {We also observe that most of the risk minimizing portfolios, especially $\mathsf{StDev}$, $\mathsf{CVaR\_N}$ and $\mathsf{CVaR\_M}$, have better risk-adjusted performance compared to the market-based ones.}
%
%
%
%
%
%
%
%
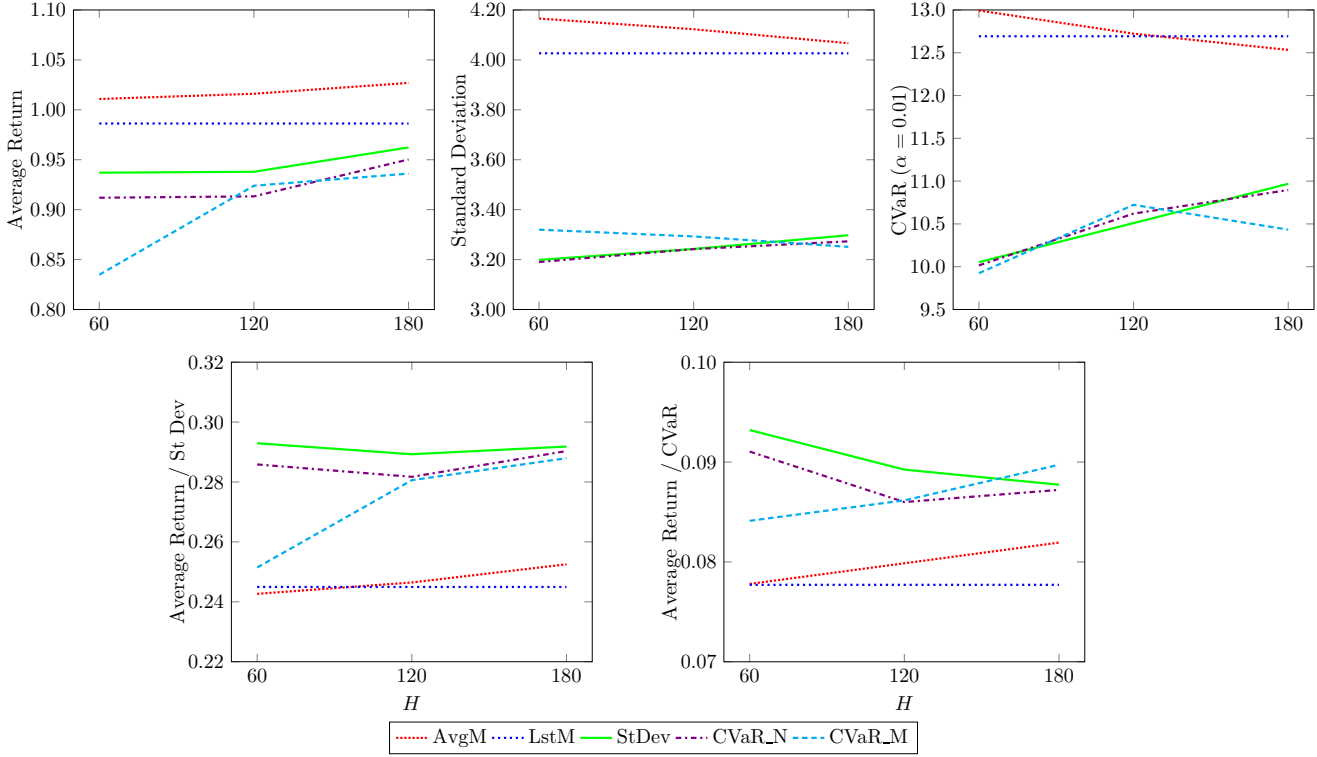
\begin{figure}[H]
\begin{tikzpicture}[scale=0.7]
\begin{groupplot}[
     group style = {group size = 3 by 2, horizontal sep=1.5cm}
]
\nextgroupplot[
    ylabel={Average Return},
    xmin=50, xmax=190,
    ymin=0.80, ymax=1.10,
    xtick={60, 120, 180},
    ytick={0.80, 0.85, 0.90, 0.95, 1.00, 1.05, 1.10},
  y tick label style={
    /pgf/number format/.cd,
    fixed,
    fixed zerofill,
    precision=2
  },
ylabel shift = -5 pt
]

\addplot[very thick,
    color=red, densely dotted, mark=, densely dotted,
    mark=,
    ]
    coordinates {
   (60,1.01083742498947)	(120,1.01606963234176)	(180,1.02692060315482)
    }; 

\addplot[very thick,
    color=blue, dotted, mark=, 
    mark=,
    ]
    coordinates {
   (60,0.986292519195984)	(120,0.986292519195984)	(180,0.986292519195984)
    }; 

\addplot[very thick,
    color=green,
    ]
    coordinates {
   (60,0.937040379778672)	(120,0.937948834235797)	(180,0.962302872529446)
    };

\addplot[very thick,
    color=violet, dashdotted, mark=,
    ]
    coordinates {
   (60,0.911936788874693)	(120,0.913323402771279)	(180,0.950258263096072)
    };

\addplot[very thick,
    color=cyan, densely dashed, mark=,
    ]
    coordinates {
(60,0.834778876833958)	(120,0.923964125950255)	(180,0.936126324997484)
    };

%

\nextgroupplot[
    ylabel={ Standard Deviation},
    xmin=50, xmax=190,
    ymin=3.00, ymax=4.20,
    xtick={60, 120, 180},
    ytick={3.00, 3.20, 3.40, 3.60, 3.80, 4.00, 4.20},
  y tick label style={
    /pgf/number format/.cd,
    fixed,
    fixed zerofill,
    precision=2
  },
ylabel shift = -5 pt,
    legend columns=7,
	legend style={at={(0.5,-0.2)},anchor=north},
]

\addplot[very thick,
    color=red, densely dotted, mark=,
    ]
    coordinates {
(60,4.16555756196506)	(120,4.1227319086464)	(180,4.06679998889105)
    }; 

\addplot[very thick,
    color=blue, dotted, mark=,
    ]
    coordinates {
(60,4.02632996543798)	(120,4.02632996543798)	(180,4.02632996543798)
    }; 

\addplot[very thick,
    color=green,
    ]
    coordinates {
(60,3.19883408388741)	(120,3.24251466666997)	(180,3.29751730299444)
    };

\addplot[very thick,
    color=violet, dashdotted, mark=,
    ]
    coordinates {
(60,3.19021019498835)	(120,3.24194918463558)	(180,3.27307419643006)
    };

\addplot[very thick,
    color=cyan, densely dashed, mark=,
    ]
    coordinates {
(60,3.31986916173798)	(120,3.29277256776533)	(180,3.2511214047799)
    };

%

\nextgroupplot[
    ylabel={CVaR ($\alpha=0.01$)},
    xmin=50, xmax=190,
    ymin=9.5, ymax=13.0,
    xtick={60, 120, 180},
    ytick={9.5,10,10.5,11,11.5,12,12.5,13},
  y tick label style={
    /pgf/number format/.cd,
    fixed,
    fixed zerofill,
    precision=1
  },
ylabel shift = -8 pt
]

\addplot[very thick,
    color=red, densely dotted, mark=,
    ]
    coordinates {
(60,12.9933299826985)	(120,12.7228370172321)	(180,12.5334045184918)
    }; 

\addplot[very thick,
    color=blue, dotted, mark=,
    ]
    coordinates {
(60,12.692639078874)	(120,12.692639078874)	(180,12.692639078874)
    }; 

\addplot[very thick,
    color=green,
    ]
    coordinates {
(60,10.0533696443415)	(120,10.5098244177614)	(180,10.9686110428349)
    };

\addplot[very thick,
    color=violet, dashdotted, mark=,
    ]
    coordinates {
(60,10.0154585615658)	(120,10.6207285437167)	(180,10.8957551997239)
    };

\addplot[very thick,
    color=cyan, densely dashed, mark=,
    ]
    coordinates {
(60,9.92411282013878)	(120,10.7232848804201)	(180,10.4332137576909)
    };

%

\nextgroupplot[
    xshift=3cm,
    xlabel={$H$},
    ylabel={Average Return / St Dev},
    xmin=50, xmax=190,
    ymin=0.22, ymax=0.32,
    xtick={60, 120, 180},
    ytick={0.22, 0.24, 0.26, 0.28, 0.30, 0.32},
  y tick label style={
    /pgf/number format/.cd,
    fixed,
    fixed zerofill,
    precision=2
  },
ylabel shift = -8 pt
]

\addplot[very thick,
    color=red, densely dotted, mark=,
    ]
    coordinates {
 (60,0.242665575964966)	(120,0.246455422000836)	(180,0.252513181361261)	
    }; 

\addplot[very thick,
    color=blue, dotted, mark=,
    ]
    coordinates {
   (60,0.244960678250993)	(120,0.244960678250993)	(180,0.244960678250993)	
    }; 

\addplot[very thick,
    color=green,
    ]
    coordinates {
(60,0.292931848043812)	(120,0.289265872526973)	(180,0.291826481594377)	
    };

\addplot[very thick,
    color=violet, dashdotted, mark=,
    ]
    coordinates {
(60,0.28585476602993)	(120,0.281720456045317)	(180,0.290325915658288)	
    };

\addplot[very thick,
    color=cyan, densely dashed, mark=,
    ]
    coordinates {
(60,0.251449330128704)	(120,0.280603687905877)	(180,0.287939516383843)	
    };

%

\nextgroupplot[
    xshift=4cm,
    xlabel={$H$},
    ylabel={Average Return / CVaR},
    xmin=50, xmax=190,
    ymin=0.07, ymax=0.10,
    xtick={60, 120, 180},
    ytick={0.07, 0.08, 0.09, 0.10},
  y tick label style={
    /pgf/number format/.cd,
    fixed,
    fixed zerofill,
    precision=2
  },
ylabel shift = -8 pt,
    legend columns=7,
	legend style={at={(-0.195,-0.2)},anchor=north},
]

\addplot[very thick,
    color=red, densely dotted, mark=,
    ]
    coordinates {
(60,0.0777966407637972)	(120,0.079861876008124)	(180,0.0819346891452917)		
    }; 
\addlegendentry{AvgM}

\addplot[very thick,
    color=blue, dotted, mark=,
    ]
    coordinates {
(60,0.0777058666103252)	(120,0.0777058666103252)	(180,0.0777058666103252)		
    }; 
\addlegendentry{LstM}

\addplot[very thick,
    color=green,
    ]
    coordinates {
(60,0.0932065976810155)	(120,0.0892449575704311)	(180,0.0877324274487842)		
    };
\addlegendentry{StDev}

\addplot[very thick,
    color=violet, dashdotted, mark=,
    ]
    coordinates {
(60,0.0910529241640756)	(120,0.0859944210994459)	(180,0.0872136208713786)		
    };
\addlegendentry{CVaR\_N}

\addplot[very thick,
    color=cyan, densely dashed, mark=,
    ]
    coordinates {
(60,0.0841162219699841)	(120,0.0861642804657128)	(180,0.0897255962293894)	
    };
\addlegendentry{CVaR\_M}

%

    \end{groupplot}
\end{tikzpicture}

\caption{Performance comparison of market-based vs. risk minimizing portfolios for different $H$ values. Vertical axes are in percentage.}
\label{fig:market vs risk-min}
\end{figure}

If we focus on the analytical approaches, we observe a general trend that larger $H$ values give rise to portfolios with a significantly higher average return, especially for $\mathsf{CVaR\_M}$, 
while the performances of $\mathsf{StDev}$ and $\mathsf{CVaR\_N}$ are less affected. The associated risks for the $\mathsf{StDev}$, $\mathsf{CVaR\_N}$ and $\mathsf{CVaR\_M}$ methods are more or less stable.
{Finally, we note that the risk-adjusted performance measures typically improve with larger $H$ values under the mixture model $\mathsf{CVaR\_M}$.

\subsection{Combining Risk Minimizing Portfolios with Market Portfolios via the Black-Litterman Approach}
\label{sec:mkt vs BL}

In the previous subsection, we had two main observations: (i) The risk minimizing approaches, in fact, do generate lower risk portfolios than market-based approaches, but also with a lower reward. (ii) The performance of all portfolios is better in terms of the average return when $H$ is large. In this subsection, we will focus on $H=180$, and apply the BL approach developed in Section~\ref{sec:extend BL CVaR}, which may result in portfolios with higher average return, lower  risk than the market-based portfolios.

Note that our previous results with $\mathsf{LstM}$, $\mathsf{AvgM}$ and $\mathsf{StDev}$
are not affected. They will serve as benchmarks in this subsection. We modify the other two portfolio construction methods as follows:

\begin{itemize}
\item  $\mathsf{CVaR\_N}$:
First, solve problem \eqref{eq:cvar normal norm} with respect to a given market capitalization $x^\text{m}$  to update the mean estimate $\mu$, and then solve problem \eqref{eq:cvar normal}  to obtain a BL portfolio.
\item  $\mathsf{CVaR\_M}$:
First, solve problem \eqref{eq:cvar mixture norm} with respect to a given market capitalization $x^\text{m}$  to update the mean estimates $\mu^i$, $i=1,2$, and then solve problem \eqref{eq:cvar mixture one-shot}  to obtain a BL portfolio.
\end{itemize}

\begin{figure}[H]
\begin{tikzpicture}[scale=0.7]
\begin{groupplot}[
     group style = {group size = 3 by 2, horizontal sep=1.5cm}
]
\nextgroupplot[
    ylabel={Average Return},
    xmin=0.0625, xmax=256,
    ymin=0.80, ymax=1.10,
    xtick={0.0625,  0.25,  1,  4, 16, 64, 256},
    xmode = log,
    log basis x=2,
    ytick={0.80, 0.85, 0.90, 0.95, 1.00, 1.05, 1.10},
  y tick label style={
    /pgf/number format/.cd,
    fixed,
    fixed zerofill,
    precision=2
  },
ylabel shift = -5 pt
]

\addplot[very thick,
    color=red, densely dotted, mark=,
    ]
    coordinates {
  (0.03125,1.02692060315482)	(0.0625,1.02692060315482)	(0.125,1.02692060315482)	(0.25,1.02692060315482)	(0.5,1.02692060315482)	(1,1.02692060315482)	(2,1.02692060315482)	(4,1.02692060315482)	(8,1.02692060315482)	(16,1.02692060315482)	(32,1.02692060315482)	(64,1.02692060315482)	(128,1.02692060315482)	(256,1.02692060315482)
    }; 

\addplot[very thick,
    color=blue, dotted, mark=,
    ]
    coordinates {
   (0.03125,0.986292519195984)	(0.0625,0.986292519195984)	(0.125,0.986292519195984)	(0.25,0.986292519195984)	(0.5,0.986292519195984)	(1,0.986292519195984)	(2,0.986292519195984)	(4,0.986292519195984)	(8,0.986292519195984)	(16,0.986292519195984)	(32,0.986292519195984)	(64,0.986292519195984)	(128,0.986292519195984)	(256,0.986292519195984)
    }; 

\addplot[very thick,
    color=green,
    ]
    coordinates {
   (0.03125,0.962302872529446)	(0.0625,0.962302872529446)	(0.125,0.962302872529446)	(0.25,0.962302872529446)	(0.5,0.962302872529446)	(1,0.962302872529446)	(2,0.962302872529446)	(4,0.962302872529446)	(8,0.962302872529446)	(16,0.962302872529446)	(32,0.962302872529446)	(64,0.962302872529446)	(128,0.962302872529446)	(256,0.962302872529446)
    };

\addplot[very thick,
    color=violet, dashdotted, mark=,
    ]
    coordinates {
   (0.03125,0.987588430960606)	(0.0625,0.988957488585163)	(0.125,0.9915753933091)	(0.25,0.996811777199706)	(0.5,1.00826284429715)	(1,0.995360708253909)	(2,0.977842587235364)	(4,0.965334758201825)	(8,0.956621139069237)	(16,0.953238776749783)	(32,0.951824247674738)	(64,0.951029594774229)	(128,0.950650922703385)	(256,0.950461260263492)
    };

\addplot[very thick,
    color=cyan, densely dashed, mark=,
    ]
    coordinates {
(0.03125,1.05562610363182)	(0.0625,1.05455160185037)	(0.125,1.05297343057421)	(0.25,1.04968521388508)	(0.5,1.04658261140533)	(1,1.04111915043152)	(2,1.02020659419801)	(4,0.991008106244282)	(8,0.967710405775277)	(16,0.953479292936577)	(32,0.944362400386483)	(64,0.939817180457363)	(128,0.937919782352262)	(256,0.937013426606585)
    };

%

\nextgroupplot[
    ylabel={Standard Deviation},
    xmin=0.0625, xmax=256,
    ymin=3.00, ymax=4.20,
    xtick={0.0625,  0.25,  1,  4, 16, 64, 256},
    xmode = log,
    log basis x=2,
    ytick={3.00, 3.20, 3.40, 3.60, 3.80, 4.00, 4.20},
  y tick label style={
    /pgf/number format/.cd,
    fixed,
    fixed zerofill,
    precision=2
  },
ylabel shift = -5 pt,
    legend columns=7,
	legend style={at={(0.5,-0.2)},anchor=north},
]

\addplot[very thick,
    color=red, densely dotted, mark=,
    ]
    coordinates {
(0.03125,4.06679998889105)	(0.0625,4.06679998889105)	(0.125,4.06679998889105)	(0.25,4.06679998889105)	(0.5,4.06679998889105)	(1,4.06679998889105)	(2,4.06679998889105)	(4,4.06679998889105)	(8,4.06679998889105)	(16,4.06679998889105)	(32,4.06679998889105)	(64,4.06679998889105)	(128,4.06679998889105)	(256,4.06679998889105)
    }; 

\addplot[very thick,
    color=blue, dotted, mark=,
    ]
    coordinates {
(0.03125,4.02632996543798)	(0.0625,4.02632996543798)	(0.125,4.02632996543798)	(0.25,4.02632996543798)	(0.5,4.02632996543798)	(1,4.02632996543798)	(2,4.02632996543798)	(4,4.02632996543798)	(8,4.02632996543798)	(16,4.02632996543798)	(32,4.02632996543798)	(64,4.02632996543798)	(128,4.02632996543798)	(256,4.02632996543798)
    }; 

\addplot[very thick,
    color=green,
    ]
    coordinates {
(0.03125,3.29751730299444)	(0.0625,3.29751730299444)	(0.125,3.29751730299444)	(0.25,3.29751730299444)	(0.5,3.29751730299444)	(1,3.29751730299444)	(2,3.29751730299444)	(4,3.29751730299444)	(8,3.29751730299444)	(16,3.29751730299444)	(32,3.29751730299444)	(64,3.29751730299444)	(128,3.29751730299444)	(256,3.29751730299444)
};

\addplot[very thick,
    color=violet, dashdotted, mark=,
    ]
    coordinates {
(0.03125,3.9398671373312)	(0.0625,3.86757300127053)	(0.125,3.7544424171383)	(0.25,3.60586443586132)	(0.5,3.47350147010741)	(1,3.38731991681144)	(2,3.33188718688703)	(4,3.30208406382403)	(8,3.28582642287496)	(16,3.27854203741762)	(32,3.27557874977694)	(64,3.27429529331945)	(128,3.27368814225987)	(256,3.27338048104396)
};

\addplot[very thick,
    color=cyan, densely dashed, mark=,
    ]
    coordinates {
(0.03125,3.69796030284234)	(0.0625,3.68229744755402)	(0.125,3.65244082077967)	(0.25,3.60087995674226)	(0.5,3.52411824344378)	(1,3.4424845555003)	(2,3.37811356820906)	(4,3.32276340787472)	(8,3.28922320606394)	(16,3.26803576139187)	(32,3.25843749910071)	(64,3.25444917792)	(128,3.25282318707778)	(256,3.25193460630771)
 };

%

\nextgroupplot[
    ylabel={CVaR ($\alpha=0.01$)},
    xmin=0.0625, xmax=256,
    ymin=9.5, ymax=13.0,
    xtick={0.0625,  0.25,  1,  4, 16, 64, 256},
    xmode = log,
    log basis x=2,
    ytick={9.5,10,10.5,11,11.5,12,12.5,13},
  y tick label style={
    /pgf/number format/.cd,
    fixed,
    fixed zerofill,
    precision=1
  },
ylabel shift = -8 pt
]

\addplot[very thick,
    color=red, densely dotted, mark=,
    ]
    coordinates {
(0.03125,12.5334045184918)	(0.0625,12.5334045184918)	(0.125,12.5334045184918)	(0.25,12.5334045184918)	(0.5,12.5334045184918)	(1,12.5334045184918)	(2,12.5334045184918)	(4,12.5334045184918)	(8,12.5334045184918)	(16,12.5334045184918)	(32,12.5334045184918)	(64,12.5334045184918)	(128,12.5334045184918)	(256,12.5334045184918)
  }; 

\addplot[very thick,
    color=blue, dotted, mark=,
    ]
    coordinates {
(0.03125,12.692639078874)	(0.0625,12.692639078874)	(0.125,12.692639078874)	(0.25,12.692639078874)	(0.5,12.692639078874)	(1,12.692639078874)	(2,12.692639078874)	(4,12.692639078874)	(8,12.692639078874)	(16,12.692639078874)	(32,12.692639078874)	(64,12.692639078874)	(128,12.692639078874)	(256,12.692639078874)
    }; 

\addplot[very thick,
    color=green,
    ]
    coordinates {
(0.03125,10.9686110428349)	(0.0625,10.9686110428349)	(0.125,10.9686110428349)	(0.25,10.9686110428349)	(0.5,10.9686110428349)	(1,10.9686110428349)	(2,10.9686110428349)	(4,10.9686110428349)	(8,10.9686110428349)	(16,10.9686110428349)	(32,10.9686110428349)	(64,10.9686110428349)	(128,10.9686110428349)	(256,10.9686110428349)
    };

\addplot[very thick,
    color=violet, dashdotted, mark=,
    ]
    coordinates {(0.03125,12.5739038239927)	(0.0625,12.4672154949281)	(0.125,12.2831211262528)	(0.25,11.9983523580055)	(0.5,11.6209555448891)	(1,11.3968346957502)	(2,11.1987534764914)	(4,11.0760510979246)	(8,10.9955943215479)	(16,10.9485561386363)	(32,10.9229311310913)	(64,10.9095757616347)	(128,10.9027057001538)	(256,10.8992442004181)
};

\addplot[very thick,
    color=cyan, densely dashed, mark=,
    ]
    coordinates {
(0.03125,12.1297625434927)	(0.0625,12.1021313149401)	(0.125,12.0490589780819)	(0.25,11.9508226152135)	(0.5,11.780647267364)	(1,11.5495991807744)	(2,11.2761850403259)	(4,10.9992615722073)	(8,10.7633906925871)	(16,10.5973346483573)	(32,10.5075032809768)	(64,10.4640507236575)	(128,10.4486182642334)	(256,10.4409999458568)
    };

%

\nextgroupplot[
    xshift=3cm,
    xlabel={$\tau$},
    ylabel={Average Return / St Dev},
    xmin=0.0625, xmax=256,
    ymin=0.22, ymax=0.32,
    xtick={0.0625,  0.25,  1,  4, 16, 64, 256},
    xmode = log,
    log basis x=2,
    ytick={0.22, 0.24, 0.26, 0.28, 0.30, 0.32},
  y tick label style={
    /pgf/number format/.cd,
    fixed,
    fixed zerofill,
    precision=2
  },
ylabel shift = -5 pt
]

\addplot[very thick,
    color=red, densely dotted, mark=,
    ]
    coordinates {
(0.03125,0.252513181361261)	(0.0625,0.252513181361261)	(0.125,0.252513181361261)	(0.25,0.252513181361261)	(0.5,0.252513181361261)	(1,0.252513181361261)	(2,0.252513181361261)	(4,0.252513181361261)	(8,0.252513181361261)	(16,0.252513181361261)	(32,0.252513181361261)	(64,0.252513181361261)	(128,0.252513181361261)	(256,0.252513181361261)
    }; 

\addplot[very thick,
    color=blue, dotted, mark=,
    ]
    coordinates {
(0.03125,0.244960678250993)	(0.0625,0.244960678250993)	(0.125,0.244960678250993)	(0.25,0.244960678250993)	(0.5,0.244960678250993)	(1,0.244960678250993)	(2,0.244960678250993)	(4,0.244960678250993)	(8,0.244960678250993)	(16,0.244960678250993)	(32,0.244960678250993)	(64,0.244960678250993)	(128,0.244960678250993)	(256,0.244960678250993)
    }; 

\addplot[very thick,
    color=green,
    ]
    coordinates {
(0.03125,0.291826481594377)	(0.0625,0.291826481594377)	(0.125,0.291826481594377)	(0.25,0.291826481594377)	(0.5,0.291826481594377)	(1,0.291826481594377)	(2,0.291826481594377)	(4,0.291826481594377)	(8,0.291826481594377)	(16,0.291826481594377)	(32,0.291826481594377)	(64,0.291826481594377)	(128,0.291826481594377)	(256,0.291826481594377)
    };

\addplot[very thick,
    color=violet, dashdotted, mark=,
    ]
    coordinates {
(0.03125,0.250665414983913)	(0.0625,0.255704931299366)	(0.125,0.264107231684458)	(0.25,0.27644183383217)	(0.5,0.29027275588456)	(1,0.293849040745719)	(2,0.29348010073203)	(4,0.292341060840197)	(8,0.291135627983731)	(16,0.290750817244549)	(32,0.29058200714593)	(64,0.290453214990907)	(128,0.290391412190881)	(256,0.290360764893535)
    };

\addplot[very thick,
    color=cyan, densely dashed, mark=,
    ]
    coordinates {
(0.03125,0.285461718672437)	(0.0625,0.286384143831417)	(0.125,0.288293084608948)	(0.25,0.291507972077674)	(0.5,0.296977155449417)	(1,0.302432482599828)	(2,0.302004824171404)	(4,0.298248170151285)	(8,0.294206365804311)	(16,0.291759136849373)	(32,0.289820627416397)	(64,0.288779184764539)	(128,0.288340228905849)	(256,0.288140304171271)
    };

%

\nextgroupplot[
    xshift=4cm,
    xlabel={$\tau$},
    ylabel={Average Return / CVaR},
    xmin=0.0625, xmax=256,
    ymin=0.07, ymax=0.10,
    xtick={0.0625,  0.25,  1,  4, 16, 64, 256},
    xmode = log,
    log basis x=2,
    ytick={0.07,0.08,0.09,0.10},
  y tick label style={
    /pgf/number format/.cd,
    fixed,
    fixed zerofill,
    precision=2
  },
ylabel shift = -5 pt,
    legend columns=7,
	legend style={at={(-0.195,-0.2)},anchor=north},
]

\addplot[very thick,
    color=red, densely dotted, mark=,
    ]
    coordinates {
(0.03125,0.0819346891452917)	(0.0625,0.0819346891452917)	(0.125,0.0819346891452917)	(0.25,0.0819346891452917)	(0.5,0.0819346891452917)	(1,0.0819346891452917)	(2,0.0819346891452917)	(4,0.0819346891452917)	(8,0.0819346891452917)	(16,0.0819346891452917)	(32,0.0819346891452917)	(64,0.0819346891452917)	(128,0.0819346891452917)	(256,0.0819346891452917)
    }; 
\addlegendentry{AvgM}

\addplot[very thick,
    color=blue, dotted, mark=,
    ]
    coordinates {
(0.03125,0.0777058666103252)	(0.0625,0.0777058666103252)	(0.125,0.0777058666103252)	(0.25,0.0777058666103252)	(0.5,0.0777058666103252)	(1,0.0777058666103252)	(2,0.0777058666103252)	(4,0.0777058666103252)	(8,0.0777058666103252)	(16,0.0777058666103252)	(32,0.0777058666103252)	(64,0.0777058666103252)	(128,0.0777058666103252)	(256,0.0777058666103252)
    }; 
\addlegendentry{LstM}

\addplot[very thick,
    color=green,
    ]
    coordinates {
(0.03125,0.0877324274487842)	(0.0625,0.0877324274487842)	(0.125,0.0877324274487842)	(0.25,0.0877324274487842)	(0.5,0.0877324274487842)	(1,0.0877324274487842)	(2,0.0877324274487842)	(4,0.0877324274487842)	(8,0.0877324274487842)	(16,0.0877324274487842)	(32,0.0877324274487842)	(64,0.0877324274487842)	(128,0.0877324274487842)	(256,0.0877324274487842)
};
\addlegendentry{StDev}

\addplot[very thick,
    color=violet, dashdotted, mark=,
    ]
    coordinates {
(0.03125,0.0785427059714068)	(0.0625,0.0793246486344517)	(0.125,0.0807266641041094)	(0.25,0.0830790551449855)	(0.5,0.0867624732236915)	(1,0.0873365925562711)	(2,0.0873170919681433)	(4,0.0871551376629805)	(8,0.0870004031700734)	(16,0.0870652499452328)	(32,0.0871400026468575)	(64,0.0871738384290505)	(128,0.0871940368609578)	(256,0.0872043274548369)
};
\addlegendentry{CVaR\_N}

\addplot[very thick,
    color=cyan, densely dashed, mark=,
    ]
    coordinates {
(0.03125,0.0870277633091949)	(0.0625,0.0871376763651973)	(0.125,0.0873905117810149)	(0.25,0.0878337205464686)	(0.5,0.0888391433554491)	(1,0.0901433144246754)	(2,0.0904744459717135)	(4,0.090097694262344)	(8,0.0899075796293219)	(16,0.0899735003729807)	(32,0.0898750516781845)	(64,0.0898138976269095)	(128,0.0897649582589163)	(256,0.0897436482583654)
 };
\addlegendentry{CVaR\_M}

%

    \end{groupplot}
\end{tikzpicture}

\caption{Performance comparison of market-based vs. risk minimizing portfolios after BL modification ($H=180$, $x^\text{m}=x_{\text{LstM}}$).}
\label{fig:market vs risk-min BL last}
\end{figure}
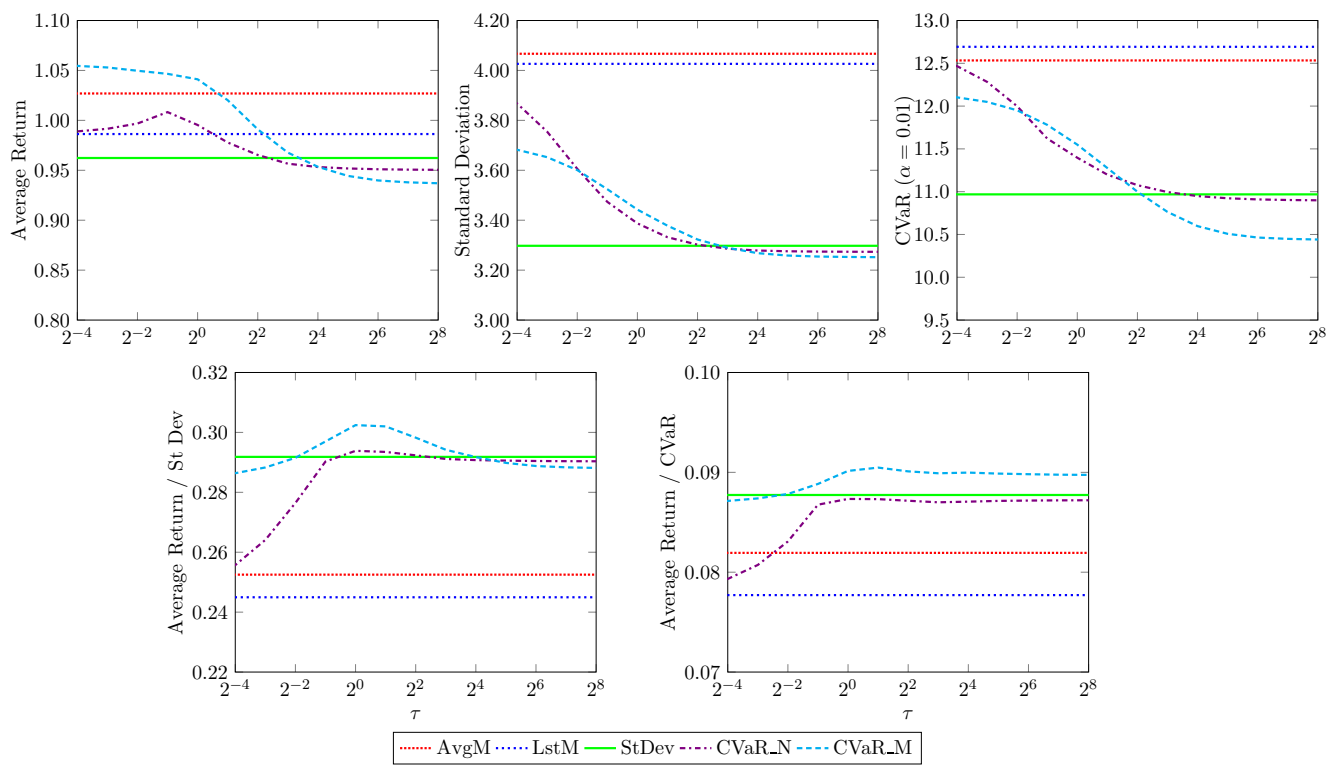

In Figure \ref{fig:market vs risk-min BL last}, we choose the last market portfolio as $x^\text{m}$ for each $t$. Although not shown in the figure, we verify that $\mathsf{CVaR\_N}$ matches $\mathsf{LstM}$ when $\tau = 0$. Interestingly, $\mathsf{CVaR\_N}$  outperforms $\mathsf{LstM}$ for $\tau \le 1$ in both average return and risk measures. Arguably, the best performing method is $\mathsf{CVaR\_M}$, which has higher average return and lower risk than $\mathsf{LstM}$ and $\mathsf{AvgM}$ when $\tau \le 4$ and $\tau \le 1$, respectively.  
{We also note the success of risk-minimizing portfolios in terms of both of the risk-adjusted performance measures, especially the  $\mathsf{CVaR\_M}$ approach.}  A final observation is that when $\tau$ is large enough, we can recover the risk minimizing portfolios as expected.

In Figure \ref{fig:market vs risk-min BL avg}, we choose the average market portfolio as $x^\text{m}$ for each $t$. We again verify that $\mathsf{CVaR\_N}$ matches $\mathsf{LstM}$ when $\tau = 0$. This time $\mathsf{CVaR\_N}$ does not perform better than $\mathsf{LstM}$ in terms of average return, but its performance is very close for $\tau \le 0.5$, with a sharp decrease in both risk measures. $\mathsf{CVaR\_M}$ is again the winner as it performs better than both market-based portfolios in similar ranges mentioned above with even better average returns. {It is also interesting to note that the $\mathsf{CVaR\_M}$ approach  outperforms the other portfolios in terms of the risk-adjusted measures for a large range of $\tau$ values considered.} 

\begin{figure}[h]
\begin{tikzpicture}[scale=0.7]
\begin{groupplot}[
     group style = {group size = 3 by 2, horizontal sep=1.5cm}
]
\nextgroupplot[
    ylabel={Average Return},
    xmin=0.0625, xmax=256,
    ymin=0.80, ymax=1.10,
    xtick={0.0625,  0.25,  1,  4, 16, 64, 256},
    xmode = log,
    log basis x=2,
    ytick={0.80, 0.85, 0.90, 0.95, 1.00, 1.05, 1.10},
  y tick label style={
    /pgf/number format/.cd,
    fixed,
    fixed zerofill,
    precision=2
  },
ylabel shift = -5 pt
]

\addplot[very thick,
    color=red, densely dotted, mark=,
    ]
    coordinates {
 (0.03125,1.02692060315482)	(0.0625,1.02692060315482)	(0.125,1.02692060315482)	(0.25,1.02692060315482)	(0.5,1.02692060315482)	(1,1.02692060315482)	(2,1.02692060315482)	(4,1.02692060315482)	(8,1.02692060315482)	(16,1.02692060315482)	(32,1.02692060315482)	(64,1.02692060315482)	(128,1.02692060315482)	(256,1.02692060315482)
    }; 

\addplot[very thick,
    color=blue, dotted, mark=,
    ]
    coordinates {
(0.03125,0.986292519195984)	(0.0625,0.986292519195984)	(0.125,0.986292519195984)	(0.25,0.986292519195984)	(0.5,0.986292519195984)	(1,0.986292519195984)	(2,0.986292519195984)	(4,0.986292519195984)	(8,0.986292519195984)	(16,0.986292519195984)	(32,0.986292519195984)	(64,0.986292519195984)	(128,0.986292519195984)	(256,0.986292519195984)
    }; 

\addplot[very thick,
    color=green,
    ]
    coordinates {
(0.03125,0.962302872529446)	(0.0625,0.962302872529446)	(0.125,0.962302872529446)	(0.25,0.962302872529446)	(0.5,0.962302872529446)	(1,0.962302872529446)	(2,0.962302872529446)	(4,0.962302872529446)	(8,0.962302872529446)	(16,0.962302872529446)	(32,0.962302872529446)	(64,0.962302872529446)	(128,0.962302872529446)	(256,0.962302872529446)
    };

\addplot[very thick,
    color=violet, dashdotted, mark=,
    ]
    coordinates {
(0.03125,1.02513784051358)	(0.0625,1.02397655534077)	(0.125,1.02255460176399)	(0.25,1.0213431606406)	(0.5,1.02069736514821)	(1,1.00434289466666)	(2,0.983660132524328)	(4,0.968595114390305)	(8,0.95848019613056)	(16,0.954274428903335)	(32,0.952375893394143)	(64,0.951307029197921)	(128,0.950792650654253)	(256,0.95053299605223)
    };

\addplot[very thick,
    color=cyan, densely dashed, mark=,
    ]
    coordinates {
(0.03125,1.07750731494944)	(0.0625,1.07678086038913)	(0.125,1.07520100490749)	(0.25,1.07245813356663)	(0.5,1.06407658354241)	(1,1.0508873275194)	(2,1.02754156962499)	(4,0.995501428666706)	(8,0.970454167858441)	(16,0.953993391065276)	(32,0.944679944287285)	(64,0.94015301006368)	(128,0.938092604692881)	(256,0.937107794993819)
    };

%

\nextgroupplot[
    ylabel={Standard Deviation},
    xmin=0.0625, xmax=256,
    ymin=3.00, ymax=4.20,
    xtick={0.0625,  0.25,  1,  4, 16, 64, 256},
    xmode = log,
    log basis x=2,
    ytick={3.00, 3.20, 3.40, 3.60, 3.80, 4.00, 4.20},
  y tick label style={
    /pgf/number format/.cd,
    fixed,
    fixed zerofill,
    precision=2
  },
ylabel shift = -5 pt,
    legend columns=7,
	legend style={at={(0.5,-0.2)},anchor=north},
]

\addplot[very thick,
    color=red, densely dotted, mark=,
    ]
    coordinates {
(0.03125,4.06679998889105)	(0.0625,4.06679998889105)	(0.125,4.06679998889105)	(0.25,4.06679998889105)	(0.5,4.06679998889105)	(1,4.06679998889105)	(2,4.06679998889105)	(4,4.06679998889105)	(8,4.06679998889105)	(16,4.06679998889105)	(32,4.06679998889105)	(64,4.06679998889105)	(128,4.06679998889105)	(256,4.06679998889105)
    }; 

\addplot[very thick,
    color=blue, dotted, mark=,
    ]
    coordinates {
(0.03125,4.02632996543798)	(0.0625,4.02632996543798)	(0.125,4.02632996543798)	(0.25,4.02632996543798)	(0.5,4.02632996543798)	(1,4.02632996543798)	(2,4.02632996543798)	(4,4.02632996543798)	(8,4.02632996543798)	(16,4.02632996543798)	(32,4.02632996543798)	(64,4.02632996543798)	(128,4.02632996543798)	(256,4.02632996543798)
    }; 

\addplot[very thick,
    color=green,
    ]
    coordinates {
(0.03125,3.29751730299444)	(0.0625,3.29751730299444)	(0.125,3.29751730299444)	(0.25,3.29751730299444)	(0.5,3.29751730299444)	(1,3.29751730299444)	(2,3.29751730299444)	(4,3.29751730299444)	(8,3.29751730299444)	(16,3.29751730299444)	(32,3.29751730299444)	(64,3.29751730299444)	(128,3.29751730299444)	(256,3.29751730299444)
};

\addplot[very thick,
    color=violet, dashdotted, mark=,
    ]
    coordinates {
(0.03125,3.97289773935135)	(0.0625,3.89443251264973)	(0.125,3.77153792691801)	(0.25,3.61083037101867)	(0.5,3.46887405650629)	(1,3.37921343811898)	(2,3.32544736544883)	(4,3.29772492688074)	(8,3.28326306637338)	(16,3.27717109844571)	(32,3.27487295800243)	(64,3.27393495648327)	(128,3.27350718754997)	(256,3.27329078477984)
};

\addplot[very thick,
    color=cyan, densely dashed, mark=,
    ]
    coordinates {
(0.03125,3.71541897003727)	(0.0625,3.69816956339538)	(0.125,3.66560663061532)	(0.25,3.61276677699696)	(0.5,3.52893079908675)	(1,3.43733429703849)	(2,3.36813272115031)	(4,3.31510426214686)	(8,3.2829441034543)	(16,3.26422824959259)	(32,3.2567663139499)	(64,3.25363311104044)	(128,3.25235299559965)	(256,3.25169941303679)
 };

%

\nextgroupplot[
    xlabel={$\tau$},
    ylabel={CVaR ($\alpha=0.01$)},
    xmin=0.0625, xmax=256,
    ymin=9.5, ymax=13.0,
    xtick={0.0625,  0.25,  1,  4, 16, 64, 256},
    xmode = log,
    log basis x=2,
    ytick={9.5,10,10.5,11,11.5,12,12.5,13},
  y tick label style={
    /pgf/number format/.cd,
    fixed,
    fixed zerofill,
    precision=1
  },
ylabel shift = -8 pt
]

\addplot[very thick,
    color=red, densely dotted, mark=,
    ]
    coordinates {
(0.03125,12.5334045184918)	(0.0625,12.5334045184918)	(0.125,12.5334045184918)	(0.25,12.5334045184918)	(0.5,12.5334045184918)	(1,12.5334045184918)	(2,12.5334045184918)	(4,12.5334045184918)	(8,12.5334045184918)	(16,12.5334045184918)	(32,12.5334045184918)	(64,12.5334045184918)	(128,12.5334045184918)	(256,12.5334045184918)
  }; 

\addplot[very thick,
    color=blue, dotted, mark=,
    ]
    coordinates {
(0.03125,12.692639078874)	(0.0625,12.692639078874)	(0.125,12.692639078874)	(0.25,12.692639078874)	(0.5,12.692639078874)	(1,12.692639078874)	(2,12.692639078874)	(4,12.692639078874)	(8,12.692639078874)	(16,12.692639078874)	(32,12.692639078874)	(64,12.692639078874)	(128,12.692639078874)	(256,12.692639078874)
    }; 

\addplot[very thick,
    color=green,
    ]
    coordinates {
(0.03125,10.9686110428349)	(0.0625,10.9686110428349)	(0.125,10.9686110428349)	(0.25,10.9686110428349)	(0.5,10.9686110428349)	(1,10.9686110428349)	(2,10.9686110428349)	(4,10.9686110428349)	(8,10.9686110428349)	(16,10.9686110428349)	(32,10.9686110428349)	(64,10.9686110428349)	(128,10.9686110428349)	(256,10.9686110428349)
    };

\addplot[very thick,
    color=violet, dashdotted, mark=,
    ]
    coordinates {(0.03125,12.4207315955228)	(0.0625,12.3195491385062)	(0.125,12.1452137636598)	(0.25,11.8765032738406)	(0.5,11.571028324409)	(1,11.3581123878368)	(2,11.1690531310267)	(4,11.0558924742213)	(8,10.9844426795969)	(16,10.9426303341016)	(32,10.9198775945883)	(64,10.9080242386748)	(128,10.9019304603668)	(256,10.8988661053918)
};

\addplot[very thick,
    color=cyan, densely dashed, mark=,
    ]
    coordinates {
(0.03125,12.0324311330812)	(0.0625,11.9996685441525)	(0.125,11.9370040527659)	(0.25,11.8482230877485)	(0.5,11.7047889324648)	(1,11.5033312154933)	(2,11.2313417986121)	(4,10.9457101467933)	(8,10.7273546257243)	(16,10.5757361045883)	(32,10.4983071519241)	(64,10.4590251204756)	(128,10.4452335014022)	(256,10.4392835494423)
    };

%

\nextgroupplot[
    xshift=3cm,
    xlabel={$\tau$},
    ylabel={Average Return / St Dev},
    xmin=0.0625, xmax=256,
    ymin=0.22, ymax=0.32,
    xtick={0.0625,  0.25,  1,  4, 16, 64, 256},
    xmode = log,
    log basis x=2,
    ytick={0.22, 0.24, 0.26, 0.28, 0.30, 0.32},
  y tick label style={
    /pgf/number format/.cd,
    fixed,
    fixed zerofill,
    precision=2
  },
ylabel shift = -5 pt
]

\addplot[very thick,
    color=red, densely dotted, mark=,
    ]
    coordinates {
(0.03125,0.252513181361261)	(0.0625,0.252513181361261)	(0.125,0.252513181361261)	(0.25,0.252513181361261)	(0.5,0.252513181361261)	(1,0.252513181361261)	(2,0.252513181361261)	(4,0.252513181361261)	(8,0.252513181361261)	(16,0.252513181361261)	(32,0.252513181361261)	(64,0.252513181361261)	(128,0.252513181361261)	(256,0.252513181361261)
    }; 

\addplot[very thick,
    color=blue, dotted, mark=,
    ]
    coordinates {
(0.03125,0.244960678250993)	(0.0625,0.244960678250993)	(0.125,0.244960678250993)	(0.25,0.244960678250993)	(0.5,0.244960678250993)	(1,0.244960678250993)	(2,0.244960678250993)	(4,0.244960678250993)	(8,0.244960678250993)	(16,0.244960678250993)	(32,0.244960678250993)	(64,0.244960678250993)	(128,0.244960678250993)	(256,0.244960678250993)
    }; 

\addplot[very thick,
    color=green,
    ]
    coordinates {
(0.03125,0.291826481594377)	(0.0625,0.291826481594377)	(0.125,0.291826481594377)	(0.25,0.291826481594377)	(0.5,0.291826481594377)	(1,0.291826481594377)	(2,0.291826481594377)	(4,0.291826481594377)	(8,0.291826481594377)	(16,0.291826481594377)	(32,0.291826481594377)	(64,0.291826481594377)	(128,0.291826481594377)	(256,0.291826481594377)
    };

\addplot[very thick,
    color=violet, dashdotted, mark=,
    ]
    coordinates {
(0.03125,0.25803277802991)	(0.0625,0.262933444607073)	(0.125,0.271124040531548)	(0.25,0.28285548078861)	(0.5,0.294244572884902)	(1,0.297212032639679)	(2,0.295797835426442)	(4,0.293716163678479)	(8,0.291929149981051)	(16,0.29118846719841)	(32,0.290813080570631)	(64,0.290569923301035)	(128,0.290450760050375)	(256,0.29039063699199)
   };

\addplot[very thick,
    color=cyan, densely dashed, mark=,
    ]
    coordinates {
(0.03125,0.29000963919249)	(0.0625,0.291165897596245)	(0.125,0.293321437146954)	(0.25,0.296852301785749)	(0.5,0.301529455839084)	(1,0.305727414533063)	(2,0.305077517632398)	(4,0.300292645402959)	(8,0.295604840434942)	(16,0.292256949612621)	(32,0.29006684951293)	(64,0.288954832329893)	(128,0.288435051780079)	(256,0.288190166420901)
    };

%

\nextgroupplot[
    xshift=4cm,
    xlabel={$\tau$},
    ylabel={Average Return / CVaR},
    xmin=0.0625, xmax=256,
    ymin=0.07, ymax=0.10,
    xtick={0.0625,  0.25,  1,  4, 16, 64, 256},
    xmode = log,
    log basis x=2,
    ytick={0.07,0.08,0.09,0.10},
  y tick label style={
    /pgf/number format/.cd,
    fixed,
    fixed zerofill,
    precision=2
  },
ylabel shift = -5 pt,
    legend columns=7,
	legend style={at={(-0.195,-0.2)},anchor=north},
]

\addplot[very thick,
    color=red, densely dotted, mark=,
    ]
    coordinates {
(0.03125,0.0819346891452917)	(0.0625,0.0819346891452917)	(0.125,0.0819346891452917)	(0.25,0.0819346891452917)	(0.5,0.0819346891452917)	(1,0.0819346891452917)	(2,0.0819346891452917)	(4,0.0819346891452917)	(8,0.0819346891452917)	(16,0.0819346891452917)	(32,0.0819346891452917)	(64,0.0819346891452917)	(128,0.0819346891452917)	(256,0.0819346891452917)
    }; 
\addlegendentry{AvgM}

\addplot[very thick,
    color=blue, dotted, mark=,
    ]
    coordinates {
(0.03125,0.0777058666103252)	(0.0625,0.0777058666103252)	(0.125,0.0777058666103252)	(0.25,0.0777058666103252)	(0.5,0.0777058666103252)	(1,0.0777058666103252)	(2,0.0777058666103252)	(4,0.0777058666103252)	(8,0.0777058666103252)	(16,0.0777058666103252)	(32,0.0777058666103252)	(64,0.0777058666103252)	(128,0.0777058666103252)	(256,0.0777058666103252)
    }; 
\addlegendentry{LstM}

\addplot[very thick,
    color=green,
    ]
    coordinates {
(0.03125,0.0877324274487842)	(0.0625,0.0877324274487842)	(0.125,0.0877324274487842)	(0.25,0.0877324274487842)	(0.5,0.0877324274487842)	(1,0.0877324274487842)	(2,0.0877324274487842)	(4,0.0877324274487842)	(8,0.0877324274487842)	(16,0.0877324274487842)	(32,0.0877324274487842)	(64,0.0877324274487842)	(128,0.0877324274487842)	(256,0.0877324274487842)
};
\addlegendentry{StDev}

\addplot[very thick,
    color=violet, dashdotted, mark=,
    ]
    coordinates {
(0.03125,0.0825344169648673)	(0.0625,0.0831180219201538)	(0.125,0.0841940390397753)	(0.25,0.0859969586241961)	(0.5,0.0882114654403755)	(1,0.0884251590732799)	(2,0.0880701453368327)	(4,0.0876089484995218)	(8,0.0872579723968056)	(16,0.0872070425270088)	(32,0.0872148872681617)	(64,0.087211671736576)	(128,0.0872132375188773)	(256,0.0872139346295841)
};
\addlegendentry{CVaR\_N}

\addplot[very thick,
    color=cyan, densely dashed, mark=,
    ]
    coordinates {
(0.03125,0.0895502582173116)	(0.0625,0.0897342169433381)	(0.125,0.0900729362371589)	(0.25,0.0905163690474052)	(0.5,0.0909095063295888)	(1,0.0913550438419098)	(2,0.0914887631459998)	(4,0.0909490033370153)	(8,0.0904653758281918)	(16,0.0902058619495417)	(32,0.0899840260545384)	(64,0.0898891626355448)	(128,0.0898105920338639)	(256,0.0897674433839746)
 };
\addlegendentry{CVaR\_M}

%
    \end{groupplot}
\end{tikzpicture}

\caption{Performance comparison of market-based vs. risk minimizing portfolios after BL modification ($H=180$, $x^\text{m}=x_{\text{AvgM}}$).}
\label{fig:market vs risk-min BL avg}
\end{figure}
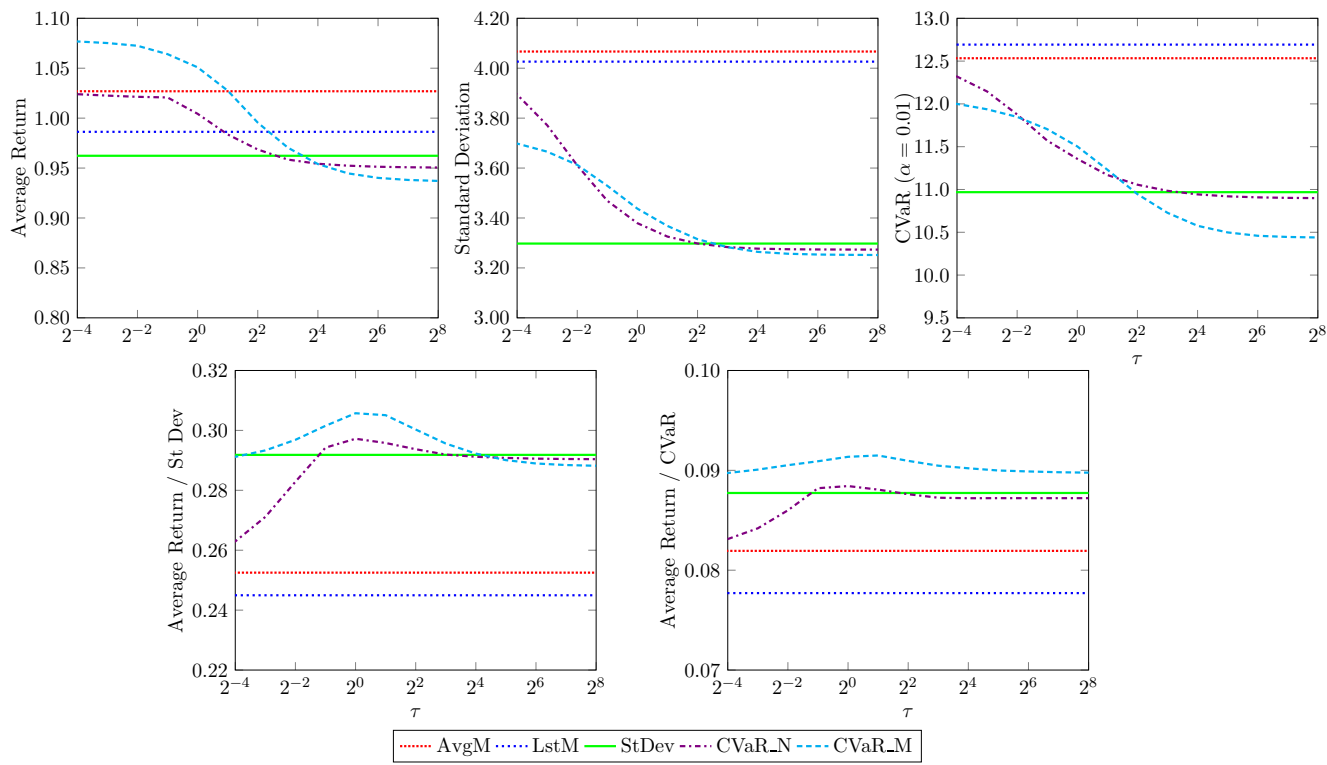

Our general conclusion regarding the experiments in this subsection is that  one can benefit from the BL modification the most if the market portfolio is chosen according to the average percentage  market capitalization and the value of $\tau$ is chosen small enough, for instance, close to $1$. This typically results in portfolios with similar average return and lower risk compared to market portfolios. Such a strategy may even lead to portfolios that dominate the market-based portfolios when the $\mathsf{CVaR\_M}$ method is used. {Moreover, the risk-minimizing portfolio constructions methods, especially the $\mathsf{CVaR\_M}$ approach, typically yield portfolios which are better in terms of the risk-adjusted measures.}

Finally, we would like to discuss the choice of the rolling horizon window $H$ in our results. When we repeated the above experiments with $H=60$ and $H=120$, we observed that $\mathsf{CVaR\_N}$ and $\mathsf{CVaR\_M}$ consistently give portfolios with similar average returns and lower risk compared to the market portfolios for small $\tau$. 
{We again observe that the risk-minimizing portfolios have better risk-adjusted performance compared to the market-based portfolios. However, it is interesting to note that the methods based on the mixture distribution typically have worse risk-adjusted performance than the ones based on the normal distribution when  $H$ is selected as 60. This is somewhat expected since the mixture models have a larger number of parameters to be estimated, hence, they may require a reasonably large dataset to have more accurate estimators than the normal models.}

\subsection{Replicating the Results on a Synthetic Dataset}
\label{sec:replication}

In the previous subsections, our experiments were based on a single stream of data, namely the monthly S\&P 500 dataset over a 30-year period. Now, we will  try to replicate these results using a synthetic dataset.
In order to carry out a reasonable simulation, we assume that the \textit{true} distribution for the return vector is a mixture of two normal random variables with the parameters given in Table~\ref{table:mix market-wise}.

\subsubsection{True Distribution is Known}
\label{sec:true known}

Under the assumption that the true distribution is known to the investors, we can obtain an \textit{ideal} CVaR minimizing portfolio by solving problem \eqref{eq:cvar mixture one-shot}, or an \textit{approximately ideal} portfolio under a normal distribution with parameters computed using \eqref{eq: r mean var} by  solving problem \eqref{eq:cvar normal}.  We compare the performance of these portfolios ($\mathsf{CVaR\_M}$ and $\mathsf{CVaR\_N}$) with an equally weighted market portfolio, that is, $x^\text{m} = \frac 1n e$ in Table~\ref{tab:true dist} and six other portfolios obtained through the BL approach.
By construction,  the $\mathsf{CVaR\_M}$ portfolio has the smallest CVaR value but its  expected return is also much smaller than that of the market portfolio. However, when combined with the market information through the BL approach,   the $\mathsf{CVaR\_M}$ portfolios retain a similar expected return while having a significantly reduced risk compared to the market portfolio. For example, the expected return of the $\mathsf{CVaR\_M}(\tau=1/4)$ portfolio is only 0.66\% smaller in relative terms than that of the equally weighted market portfolio whereas the standard deviation and CVaR measures are improved by about 9.1\% and 11.6\%, respectively. This is also reflected in the risk-adjusted performance measures.

We observe that the $\mathsf{CVaR\_N}$ approach, which is obtained via a normal approximation, also produces a lower-reward, lower-risk portfolio compared to the market portfolio. Interestingly, its average reward is slightly better than the $\mathsf{CVaR\_M}$ portfolio with higher risk in terms of the CVaR measure. We also see that the performance of the $\mathsf{CVaR\_N}$ portfolios with the BL modification is very similar compared to their $\mathsf{CVaR\_M}$ counterparts, especially when $\tau \ge 1/4$, suggesting that the BL approach may not be very sensitive to the exact distribution.
\begin{table}[h]\small
\centering
\begin{tabular}{c|ccc|cc}
\hline
           &        Avg &      St Dev &      1\% CVaR  &         Avg/St Dev &   Avg/1\% CVaR\\
\hline
    Market &      1.31  &       4.17 &      13.01 &                                0.31&	0.10\\
\hline
$\mathsf{CVaR\_M}(\tau=1/16)$&      1.31 &       3.86 &      11.79&        0.34	&0.11 \\
$\mathsf{CVaR\_M}(\tau=1/4)$ &      1.30 &       3.79 &      11.50 &         0.34&	0.11\\
$\mathsf{CVaR\_M}(\tau=1)$ &      1.29 &       3.61 &      10.67  &           0.36	&0.12\\
        $\mathsf{CVaR\_M}$ &      1.20 &       3.41 &       9.38 &               0.35&	0.13\\
\hline
$\mathsf{CVaR\_N}(\tau=1/16)$ &1.31 &	4.04 &	12.50&      0.32&	0.10\\
$\mathsf{CVaR\_N}(\tau=1/4)$ &1.31 &	3.79 &	11.52&       0.34	&0.11\\
$\mathsf{CVaR\_N}(\tau=1)$ &1.29 &	3.52 &	10.33&  	 0.37	&0.12\\
$\mathsf{CVaR\_N} $&1.24 &	3.39 &	9.62&         	0.37&	0.13\\
\hline
\end{tabular}
\caption{Performance comparison of market-based, risk minimizing and BL-type portfolios under the true distribution. The average return (Avg), standard deviation (St Dev), 1\% CVaR and two risk-adjusted performance measures are reported.}\label{tab:true dist}
\end{table}


Because we assume here that the true distribution of the returns is known, it is interesting to observe what happens when we add an expected return constraint of the form $\hat \mu^T x \ge \mu_0$ to the CVaR minimization problems. In Table \ref{tab:true dist thresh}, we report the performance of the optimal portfolios under this additional constraint with two different values of $\mu_0$. These results demonstrate that one can obtain portfolios with higher expected returns and smaller risk than the market portfolio under the assumption that the true distribution is known. We once again observe the similarity of the performances of the $\mathsf{CVaR\_M}$ and $\mathsf{CVaR\_N}$ portfolios.
\begin{table}[h]\small
\centering
\begin{tabular}{cc|ccc}

\hline
    &       &        Avg &      St Dev &      1\% CVaR \\
\hline
\multirow{ 2}{*}{$\mathsf{CVaR\_M}$}  & $\mu_0 = \hat\mu^T x^\text{m}$  & 1.31 &      3.49 &      9.67 \\
& $\mu_0 = 1.05\hat\mu^T x^\text{m}$   &  1.38 &      3.65 &     10.22 \\
\hline
\multirow{ 2}{*}{$\mathsf{CVaR\_N}$}  & $\mu_0 = \hat\mu^T x^\text{m}$  & 1.31 &      3.46 &      9.86 \\
& $\mu_0 = 1.05\hat\mu^T x^\text{m}$   &  1.38 &      3.62 &     10.37 \\
\hline
%
%
\end{tabular}
\caption{Performance of risk minimizing  portfolios under the true distribution with an expected return constraint.}\label{tab:true dist thresh}
\end{table}

\subsubsection{True Distribution is not Known}
\label{sec:true not known}

So far in this subsection, we assumed that the true distribution of the returns is known. Of course this is not the case in reality. We now relax this assumption and design the following experiment: Suppose that we are given 181 random return vectors drawn from the mixture distribution. We use the first 180 of these vectors to estimate the parameters of the true distribution under the mixture  and normal models. Based on these estimated parameters we  obtain the $\mathsf{CVaR\_M}$ and $\mathsf{CVaR\_N}$  portfolios and their BL versions with varying $\tau$ values using the optimization algorithms described earlier. Finally, we evaluate these portfolios together with the equally weighted market portfolio using the last return vector.  We repeat this experiment  10000 times and report the statistics in Table~\ref{tab:true dist10000}. We again note that the BL version of the $\mathsf{CVaR\_M}$ approach yields portfolios with practically the same average return as the market portfolio for all the values of $\tau$ considered with a significantly reduced risk. For example, the average return of the $\mathsf{CVaR\_M}(\tau=1)$ portfolio is only 0.40\% lower than that of the market portfolio in relative terms whereas the standard deviation and 1\% CVaR measures are improved by about 12.4\% and 17.4\%, respectively. We note that the reduction in the CVaR measure is even more dramatic when smaller values of $\alpha$ are considered. We also point out that the statistics of the $\mathsf{CVaR\_M}$ and $\mathsf{CVaR\_N}$ portfolios with the BL modification   are almost indistinguishable for $\tau \ge 1/4$ while the performance of the $\mathsf{CVaR\_M}(\tau=1/16)$ portfolio is slightly better than that of the $\mathsf{CVaR}(\tau=1/16)$ portfolio.
\begin{table}[h]\footnotesize
\centering
\begin{tabular}{c|ccccc|cc}
\hline
           &        Avg &      St Dev &   1\% CVaR & 0.1\% CVaR & 0.05\% CVaR    &        Avg/St Dev &  Avg/1\% CVaR\\
\hline
    Market &       1.31 &       4.08 &      12.46 &      17.59 &      18.74&        0.32	&0.10	 \\
\hline
$\mathsf{CVaR\_M}(\tau=1/16)$ &       1.31 &       3.80 &      11.27 &      15.79 &      16.73&	0.35	&0.12 \\

$\mathsf{CVaR\_M}(\tau=1/4)$ &       1.31 &       3.72 &      10.94 &      15.34 &      16.10& 0.35	&0.12 \\

$\mathsf{CVaR\_M}(\tau=1)$ &       1.30 &       3.57 &      10.29 &      14.36 &      15.10  &     	0.36	&0.13  \\

$\mathsf{CVaR\_M}$ &       1.24 &       3.49 &       9.75 &      13.90 &      14.78&     	0.35	&0.13  \\
\hline
$\mathsf{CVaR\_N}(\tau=1/16)$&1.31&	3.95	&11.91&	16.90&	17.88& 0.33	&0.11	\\
$\mathsf{CVaR\_N}(\tau=1/4)$&1.31	&3.72	&10.94&	15.53&	16.25& 0.35	&0.12	\\
$\mathsf{CVaR\_N}(\tau=1)$&1.29&	3.51	&10.03&	14.16&	14.88& 0.37&	0.13\\
$\mathsf{CVaR\_N}$&1.25	&3.43&	9.58&	13.69	&14.52&      0.37	&0.13 \\

\hline
\end{tabular}
\caption{Performance comparison of market-based, risk minimizing and BL-type portfolios under the synthetic data with 10000 replications.}\label{tab:true dist10000}
\end{table}

%
%
%
%
%

As a final test, we repeat the above experiment with an expected return constraint and report the results  in Table~\ref{tab:true dist10000 thresh}. When compared to the case without the expected return constraint (Table~\ref{tab:true dist10000}), the $\mathsf{CVaR\_M}$ portfolio has   higher average return and higher risk. Also, its performance is similar to the market portfolio in terms of the average return and it has a smaller risk. Nevertheless,  a BL modified $\mathsf{CVaR\_M}$ portfolio without the expected return constraint, such as $\mathsf{EM}(\tau=1/4)$, would still be preferable.

On the other hand,   the $\mathsf{CVaR\_N}$ portfolio with an expected return constraint produces counterintuitive results: Its average return is smaller compared to the case without the expected return constraint and the risk is higher. This observation suggests that when the true parameters are not known and an additional expected return constraint is enforced,  the $\mathsf{CVaR\_N}$ approach, which assumes a normal distribution for the return vector, may produce poor results and its behavior can be fundamentally different than that of  the $\mathsf{CVaR\_M}$ approach, which assumes a mixture distribution.
\begin{table}[h]\small
\centering
\begin{tabular}{c|ccc}

\hline
           &        Avg &      St Dev &      1\% CVaR \\
\hline
{$\mathsf{CVaR\_M}$}   &1.31&	3.83&	11.12 \\
\hline
{$\mathsf{CVaR\_N}$}   &1.22 &	3.52	&9.98 \\
\hline
%
%
\end{tabular}
\caption{Performance of risk minimizing  portfolios under the synthetic data with 10000 replications with an expected return constraint $\hat \mu^T x \ge  \hat\mu^T x^\text{m}$.}\label{tab:true dist10000 thresh}
\end{table}

\section{Conclusion}
\label{sec:conclusion}

In this paper, we  addressed  the basic question of portfolio construction, combining analytical and market-based approaches. As an important component of the analytical approaches, we proposed to model stock returns as the mixture of two normals. We then explored the possibility of constructing risk minimizing portfolios under different probabilistic models (including the mixture distribution) and risk measures (including Conditional Value-at-Risk). We also proposed a Black-Litterman type approach using an inverse optimization framework to incorporate market information into a portfolio that minimizes the Conditional Value-at-Risk under the mixture distribution. Our computational experiments showed that market portfolios typically yield high-reward, high-risk portfolios, and that risk can be  reduced by  combining analytical and market-based approaches while achieving similar average returns.

\section*{Acknowledgments}
We wish to thank Dr. Rob Stubbs for his insightful suggestions on an earlier version of this paper, {\black and two anonymous reviewers for their constructive comments which have helped us to improve our paper}.
This work was supported in part by NSF grant CMMI1560828 and ONR grant N00014-12-10032.

\bibliographystyle{apacite}
\bibliography{references2}

\end{document}